\documentclass[11pt]{amsart}
\usepackage{amsmath,amsthm,latexsym}
\usepackage{enumerate}
\usepackage{amssymb,amsfonts,mathrsfs}
\usepackage{multirow,longtable}
\usepackage{float,appendix}
\usepackage{amscd, pb-diagram}
\usepackage[all,cmtip]{xy}
\usepackage{tikz}

\usepackage{array}
\usepackage{appendix}

\usepackage{verbatim,bbm}
\usepackage{dynkin-diagrams}

\usepackage{cleveref}

\usepackage{fullpage}

\allowdisplaybreaks

\makeatletter
\def\imod#1{\allowbreak\mkern10mu({\operator@font mod}\,\,#1)}
\makeatother

\makeatletter
\renewcommand\section{\@startsection{section}{1}{\z@}%
                                  {-3.5ex \@plus -1ex \@minus-.2ex}%
                                  {2.3ex \@plus.2ex}%
                                  {\center\normalfont\large\bfseries}}
\makeatother

\theoremstyle{plain}
\newtheorem{Thm}{Theorem}[section]
\newtheorem{Cor}[Thm]{Corollary}
\newtheorem{Lem}[Thm]{Lemma}
\newtheorem{Ex}[Thm]{Example}

\theoremstyle{definition}
\newtheorem{Remark}[Thm]{Remark}

\theoremstyle{plain}

\numberwithin{equation}{section}
\usepackage[colorinlistoftodos,prependcaption,textsize=footnotesize,textwidth=0.7in]{todonotes}

\newcounter{todocounter}
\newcommand{\todonum}[1]{\stepcounter{todocounter}\todo{\thetodocounter: #1}}
\makeatletter
\providecommand\@dotsep{5}
\renewcommand{\listoftodos}[1][\@todonotes@todolistname]{%
  \@starttoc{tdo}{#1}}
\makeatother

\DeclareMathAlphabet{\mathpzc}{OT1}{pzc}{m}{it}

\usepackage{fancyhdr}
\fancypagestyle{plain}{%
        \fancyhf{} 
        \fancyfoot[C]{\normalfont \thepage} 

}

\newcounter{qcounter}

\newcommand{\C}{\mathbb{C}}
\newcommand{\Z}{\mathbb{Z}}

\DeclareMathOperator{\Hom}{Hom}
\DeclareMathOperator{\Image}{Im}

\DeclareMathOperator{\Ind}{Ind}
\DeclareMathOperator{\Ker}{Ker}

\DeclareMathOperator{\rk}{rank}
\DeclareMathOperator{\St}{St}

\newcommand{\Card}[1]{\left\vert #1\right\vert} 
\DeclareMathOperator{\zfun}{\zeta} 
\newcommand{\coset}[1]{\left[ #1 \right]}  
\newcommand{\FNorm}[1]{\left\vert #1 \right\vert} 

\newcommand{\R}{\mathbb{R}}
\newcommand{\N}{\mathbb{N}}

\newcommand{\bk}[1]{\left(#1\right)} 
\newcommand{\bm}{\begin{multline*}}
\newcommand{\tu}{\end{multline*}}

\DeclareMathOperator{\Id}{\mathbf{1}} 
\newcommand{\Gm}{\mathbb{G}_m} 
\newcommand{\modf}[1]{\mathcal{\delta}_{#1}} 

\renewcommand{\check}[1]{#1 ^{\vee}} 

\newcommand{\set}[1]{\left\{ #1 \right\}} 
\newcommand{\mvert}{\mathrel{}\middle\vert\mathrel{}} 
\newcommand{\res}[1]{\Bigg\vert_{#1}}

\newcommand{\lmod}{\backslash}

\newcommand{\gen}[1]{\left< #1 \right>}
\newcommand{\Stab}{\operatorname{Stab}}

\newcommand{\intl}{\, \int\limits}
\newcommand{\suml}{\, \sum\limits}
\newcommand{\prodl}{\, \prod\limits}

\newcommand{\restline}{\mathcal{S}}

\makeatletter
\newcommand{\newreptheorem}[2]{\newtheorem*{rep@#1}{\rep@title}\newenvironment{rep#1}[1]{\def\rep@title{#2 \ref*{##1}}\begin{rep@#1}}{\end{rep@#1}}}
\makeatother

\newreptheorem{cor}{Corollary}
\newreptheorem{theorem}{Theorem}
\newreptheorem{lemma}{Lemma}
\newreptheorem{proposition}{Proposition}

\renewcommand\AA{\mathbb{A}} 
\newcommand\BB{\mathbb{B}}

\newcommand\GG{\mathbb{G}}

\newcommand\NN{\mathbb{N}}

\newcommand\PP{\mathbb{P}}

\newcommand\TT{\mathbb{T}}

\DeclareMathAlphabet{\mathcal}{OMS}{cmsy}{m}{n}

\newcommand\cO{\mathcal{O}}

\newcommand\fraka{\mathfrak{a}}

\newcommand\GA{{\GG\bk{\AA}}}
\newcommand\BA{{\BB\bk{\AA}}}

\newcommand\IPoG{\Ind_{P_0}^{G}}
\newcommand\IPoGl{\IPoG\lambda}
\newcommand\IBG{\Ind_B^G}
\newcommand\IBGl{\IBG\lambda}
\newcommand\IBGlo{\IBG\lambda_0}
\newcommand\IPG{\Ind_P^{G}}

\newcommand\Homc{\Hom_{\textrm{cts}}}
\newcommand\Homu{\Hom_{\textrm{ur}}}
\newcommand\ars{\fraka_\R^\ast}
\newcommand\acs{\fraka_\C^\ast}
\newcommand\aMcs{\fraka_{M,\C}^\ast}
\newcommand\tr{\operatorname{\mathbf{1}}}  

\title[Singrularities and Principal Series Representations]{Singularities of Intertwining Operators and Decompositions of Principal Series Representations}
\author{Taeuk Nam${}^{1}$}
\email{taeuk.nam@alumni.ubc.ca}
\author{Avner Segal${}^{1,2}$}
\email{avners@post.tau.ac.il}
\author{Lior Silberman${}^{1}$}
\email{lior@math.ubc.ca}
\address{(1) Department of Mathematics, University of British Columbia, Vancouver, BC, V6T 1Z2, Canada}
\address{(2) Department of Mathematics, Bar Ilan University, Ramat Gan 5290002, Israel}

\begin{document}

\begin{abstract}
In this paper, we show that, under certain assumptions,
a parabolic induction $\IBGl$ from the Borel subgroup $B$ of
a (real or $p$-adic) reductive group $G$ decomposes into a direct sum
of the form:
\[
\IBGl = \bk{\IPG \St_M\otimes \chi_0} \oplus \bk{\IPG \tr_M\otimes \chi_0},
\]
where $P$ is a parabolic subgroup of $G$ with Levi subgroup $M$ of
semi-simple rank $1$, $\tr_M$ is the trivial representation of $M$,
$\St_M$ is the Steinberg representation of $M$ and $\chi_0$ is a certain
character of $M$.
We construct examples of this phenomenon for all simply-connected simple
groups of rank at least $2$.
\end{abstract}

\subjclass[2010]{22E50, 47G10, 22E46}

\maketitle

\tableofcontents

\section{Introduction}

Fixing our notation, let $F$ be a local field, $\GG$ a reductive $F$-group.
We write $G = \GG(F)$ in the analytic topology, and more generally use
roman letters to denote the set of $F$-points of the correpsonding
algebraic subgroup.  Accordingly let $P_0\subset G$ be a minimal parabolic
subgroup (formally $P_0 = \PP_0(F)$ where $\PP_0 \subset \GG$ is a minimal
parabolic subgroup defined over $F$, and similarly for other subgroups),
and let $T\subset P_0$ be a Levi subgroup.
The principal series of representations of $G$ consists of
the admissible representations $\IPoGl$ (normalized induction) as $\lambda$
varies over the characters $\Homc\bk{T,\C^\times}$.

Understanding the structure of these representations is a basic problem in the
representation theory of $G$.  Common questions about the structure include:
\begin{itemize}
	\item Is $\IPoGl$ reducible?
	\item What is the length of its composition series?
	\item What are the composition factors?  At least the irreducible subrepresentations and quotients?
	\item What is the composition series?
\end{itemize}


We specialize to the case of a quasi-split Chevalley group $\GG$
defined over $F$, in which $P_0 = B$ is a Borel subgroup and $T$ is
a maximal torus of $B$ of maximal split $F$-rank.
We may as well assume $\rk\bk{G}>1$.
Let $\acs = X^\ast\bk{T}\otimes\C^\times=\Homu\bk{T,\C^\times}$
be the space of unramified quasicharacters of $T$.

Fixing $\lambda_0\in\acs$, we study the induced representation $\IBGl$.
We prove (\Cref{Thm_Main}) that, under certain assumptions on $\lambda_0$,
the representation $\IBGl$ decomposes as the direct sum
\begin{equation}
\label{eq:introduction_decomposition}
\IBGlo = \bk{\IPG \St_M\otimes \chi_0} \oplus \bk{\Ind_P^G \tr_M\otimes \chi_0},
\end{equation}
where:
\begin{itemize}
\item $\PP$ is a parabolic subgroup of $G$ with Levi subgroup $M$ of semi-simple rank $1$.
\item $\tr_M$ (resp.\ $\St_M)$ is the trivial (resp.\ Steinberg) representation of $M$.
\item $\chi_0$ is a character of $M$ associated to the induction in stages from $B$ to $M$.
\end{itemize}

In fact, \Cref{Thm_Main} identifies the two invariant subspaces isomorphic to
$\IPG \St_M\otimes \chi_0$ and $\IPG \tr_M\otimes \chi_0$ as eigenspaces
of a certain intertwining operator.  Furthermore, this shows that each of
the two admits a unique irreducible subrepresentation.

This decomposition is rather surprising since, for generic $\chi_0$, and the associated $\lambda_0\in\mathfrak{a}_\C^\ast$, only one of the following exact sequences hold
\[
\begin{split}
& \Ind_P^G \St_M\otimes \chi_0 \hookrightarrow \Ind_B^G \lambda_0 \twoheadrightarrow \Ind_P^G \tr_M\otimes \chi_0 \\
& \Ind_P^G \tr_M\otimes \chi_0 \hookrightarrow \Ind_B^G \lambda_0 \twoheadrightarrow \Ind_P^G \St_M\otimes \chi_0 .
\end{split}
\]

The reason that these sequences split as in \Cref{eq:introduction_decomposition} is that $\lambda_0$ lies in the intersection between two singularities of a certain standard intertwining operator $N\bk{w,\lambda}$.
Namely, $N\bk{w,\lambda}$ admits a simple "pole" along a hyperplane $H_1$ and a simple "zero" along a hyperplane $H_2$ such that $\lambda_0 \in H_1\cap H_2$.
In such a case $N\bk{w,\lambda_0}$ is not well defined.
However, we show the existence of a line $\restline$ along which $N\bk{w,\lambda}$ is well-defined and continuous at $\lambda_0$.
The limit of $N\bk{w,\lambda}$ at $\lambda_0$ along $\restline$ is an intertwining operator $E$ of $\Ind_B^G \lambda_0$.
Furthermore, $\Ind_P^G \St_M\otimes \chi_0\oplus\Ind_P^G \tr_M\otimes \chi_0$ is a decomposition of $\Ind_B^G \lambda_0$ into eigenspaces of $E$.

\vspace{0.5cm}

In \Cref{Sec:Examples}, we find an abundant amount of points where the assumptions of \Cref{Thm_Main} are satisfied.
We find distinct such $\lambda_0$ for every $G$ and every Levi subgroup $M$ as above.
In fact, when $rank\bk{G}>2$, we show the existence of infinitely many such $\lambda_0$ (see \Cref{Thm:Key_Example}).
In particular, one has (\Cref{Cor:Key_Example}):

	\noindent\textbf{\Cref{Cor:Key_Example}}
	{\it For any simple group $G$ and any simple root $\alpha$, let $w_\alpha \in W$ be the corresponding simple reflection in the Weyl group and let $\omega_\alpha$ be the associated fundamental weight.
	Let $\lambda_0=-w_\alpha\cdot\omega_\alpha$.
	Then
	\begin{equation}
	\Ind_B^G {\lambda_0} = 
	\bk{\Ind_{P}^G \St_M\otimes \chi_0} \oplus 
	\bk{\Ind_{P}^G \tr_M\otimes \chi_0} .
	\end{equation}}

We note here that \Cref{eq:introduction_decomposition} implies that
\begin{equation}
\label{eq:contragredient_decomposition}
\Ind_B^G \bk{-\lambda_0} = \bk{\Ind_P^G \St_M\otimes \bk{-\chi_0}} \oplus \bk{\Ind_P^G \tr_M\otimes \bk{-\chi_0}},
\end{equation}
where we use additive notation for $\acs$.
This, again, is a decomposition into eigenspaces of the limit of $N\bk{w^{-1},\lambda}$ at $-\lambda_0$.
In particular, each of $\Ind_P^G \St_M\otimes \bk{-\chi_0}$ and $\Ind_P^G \tr_M\otimes \bk{-\chi_0}$ admits a unique irreducible quotient and it is easy to find the Langlands operator (in the sense of \cite[Cor. 4.6]{MR1721403} or \cite[Cor. 3.2]{MR2050093}) for each.

One possible application to the results of this paper is to the
computation of the residual spectrum of adelic groups.
Namely, the irreducible subrepresentations of $\IBGl$ can appear as
local constituents of residual representations of $\GA$.
In particular, the eigenvalue of the intertwining opertor on
$\IPG \St_M\otimes \chi_0$ which appears in the proof of \Cref{Thm_Main}
dictates which irreducible subrepresentation of $\Ind_\BA^\GA(\lambda)$
will appear in the residual spectrum.

Such considerations have appeared in the computation of the residual
spectrum of $Sp_4$ (see \cite{MR3267117}), $G_2$ (see \cite{MR1426903}
and \cite{Zampera1997}) and quasi-split forms of $Spin_8$
 (see \cite{LaoResidualSpectrum} and \cite{SegalEisen,SegalResiduesD4}).
It is interesting to note that when $\GG = Sp_4$ the unramified local
constituents appear only in the non-square-integrable automorphic
spectrum as can be seen by comparing \cite[Theorem 5.4]{MR1351833}
with \cite[Theorem 3.6(1)]{MR3267117}.

%
%
%
%
%

This paper is organized as follows:
\begin{itemize}
	\item In \Cref{Section:Preliminaries} we discuss the assumptions we make on the group $G$ and recall the definition and basic properties of the normalized intertwining operators used in this paper.
	\item In \Cref{Section:Levi_Subgroups_of_SS_rank_1} we proof the main result of this paper (\Cref{Thm_Main} and \Cref{Cor:SSRank1WithoutLanglandsPosition}).
	\item In \Cref{Sec:Examples} we study a family of examples of points $\lambda_0$ for which \Cref{Thm_Main} holds. In particular, for any simple group $G$ and any simple root with respect to $T$ we construct a different point $\lambda_0$ which satisfy the assumptions of \Cref{Thm_Main}.
	\item In \Cref{Section:Higher_rank_Levis} we discuss a generalization of \Cref{Thm_Main} and \Cref{Thm:Key_Example} for decompositions with respect to larger Levi subgroups $M$.
	\item In \Cref{Appendix:Facts_on_roots_systems_and_Weyl_groups} we prove a few simple results which did not fit into the body of the paper.
\end{itemize}

\subsection*{Acknowledgements}
This research was supported by the Israel Science Foundation, grant number 421/17 (Segal).
\section{Notation and Preliminaries}\label{Section:Preliminaries}

\subsection{Algebraic groups}
Let $F$ be a local field of characteristic $0$.
Let $\GG$ be a semi-simple group over $F$.

It is known (see the next section) that the following assumption guarantees
certain anayltic properties of normalized intertwining operators.  Accordingly,
while our results likely hold in greater generality we suppose that:
\begin{itemize}
\item If $F$ is Archimedean, assume that $\GG$ is a connected, quasi-split, semi-simple, linear Lie group.
\item If $F$ is non-Archimedean, assume that $\GG$ is a semi-simple Chevalley
group in the sense of \cite[pg. 21]{MR0466335}.
\end{itemize}

The papers \cite[Theorem 5.3]{MR597811} (Archimedean case) and
\cite[Theorem 6.1, p.\ 953]{MR517138} ($p$-adic case) determine the analytic
behaviour of normalized intertwining operators under the hypotheses above.
We believe that these necessary properties hold in greater generlaity;
in any case the assumption on $G$ could be replaced with hypotheses on the
anlytic behaviour of the intertwining operators.

Fix a Borel subgroup and a maximpla $F$-split torus $\GG\supset\BB\supset\TT$.
Also, let $\NN \subset \BB$ be the unipontent radical and let $G = \GG(F)$,
$B = \BB(F)$, $T = \TT(F)$, $N = \NN(F)$.

Let $\Phi=\Phi(\GG:\TT)$ be the set of roots of $\GG$ with respect to $\TT$,
$\Phi^{+}$ the roots occuring in $\NN$, that is the positive roots with respect
to the choice of $\BB$.  Let $\Delta\subset\Phi^{+}$ be the corresponding
set of simple roots.  We denote the relative semisimple rank of $\GG$ by
$n=\Card{\Delta}$.

Recall that $N_G(T)$ surjects onto the Weyl group 
$W = W\bk{\GG:\TT} = N_\GG\bk{\TT}/C_\GG\bk{\TT}$, which is generated by the
involutions $\{s_\alpha\}_{\alpha \in \Delta}$.

Let $X^\ast\bk{T}=\Hom_{F}\bk{\TT,\Gm}\cong \Z^n$ denote the group of
$F$-rational characters of $\TT$.
Let $\ars=X^\ast\bk{T}\otimes_\Z \R =\Homu\bk{T,\R^\times}$ be the space of
unramified real characters of the topological group $T$ and let
$\acs=X^\ast\bk{T}\otimes_\Z \C=\Homu\bk{T,\C^\times}$ be the space of
unramified complex characters of $T$.
The set of fundamental weights
$\set{\omega_\alpha\mvert \alpha\in\Delta} \subset X^\ast\bk{T}$ given by
$\gen{\omega_\alpha,\check{\beta}}=\delta_{\alpha,\beta}$, is basis for
$\ars$, hence gives an identification $\ars\cong\R^n$ and $\acs\cong\C^n$
as vector spaces:
\begin{equation}
\lambda=\bk{s_1,...,s_n} \mapsto \suml_{i=1}^n s_i\cdot\omega_{\alpha_i} .
\end{equation}

Finally we recall the correspondence
\[\begin{array}{ccc}
\set{\Theta\subseteq\Delta} & \longleftrightarrow & \set{\begin{matrix}
\text{Standard parabolic} \\ \text{subgroups of $G$}
\end{matrix}} \\
\Theta & \longrightarrow & P_\Theta=M_\Theta U_\Theta \\
\Delta_M & \longleftarrow & P=MU .
\end{array}\]

For a Levi subgroup $M$ of $G$, let
\begin{equation}
\aMcs=X^\ast\bk{M}\otimes_\Z\C=\Homu\bk{M,\C^\times}.
\end{equation}

Let $K\subset G$ be a maximal compact subgroup (specifically the group
$\GG(\cO_F)$ when $F$ is non-Archimedean, and recall the
\emph{Iwasawa decomposition} $G = PK$ for all parabolic subgroups $P$.

\subsection{Representation Theory and Intertwining Operators}
For any reductive group $M$ we write $\tr_M$ for the trivial representation of $M$.

For $\lambda\in\acs$ we write $\IBGl$ for the (normalized) induction of
$\lambda$ (thought of as a character of $B$) to $G$.  Recall that for all
$w\in W$ we have an intertwining operator
\[
M\bk{w,\lambda}\colon \IBGl \to \IBG\bk{w\cdot\lambda}
\]
defined by analytic continuation of the following integral (which converges
absolutely in the positive Weyl chamber)
\[
M\bk{w,\lambda}f_\lambda\bk{g} = \intl_{N\cap w Nw^{-1}\lmod N} f_\lambda\bk{w^{-1}ug} du .
\]

We collect here some necessary results regarding the intertwining operators;
a more detailed discussion may be found in \cite[sec. 3]{SegalEisen} or \cite[sec. 3]{SegalResiduesD4}.
\begin{itemize}
	\item (\emph{Gindikin--Karpelevich formula})
	Let $f_\lambda^0\in \IBGl$ denote the spherical ($K$-invariant) vector,
	normalized so that $f_\lambda^0\bk{1}=1$.
	Then
	\begin{equation}
	M\bk{w,\lambda}f_\lambda^0 = \bk{\prod_{\alpha>0,\ w\cdot\alpha<0} \frac{\zfun\bk{\gen{\lambda,\check{\alpha}}}}{\zfun\bk{\gen{\lambda,\check{\alpha}}+1}}} f_{w\cdot\lambda}^0 ,
	\end{equation}
	where $\zfun\bk{s}$ is the local $\zeta$-function of $F$.
	
	\item The operators
	\begin{equation}
	N\bk{w,\lambda} = \bk{\prod_{\alpha>0,\ w\cdot\alpha<0} \frac{\zfun\bk{\gen{\lambda,\check{\alpha}}+1}}{\zfun\bk{\gen{\lambda,\check{\alpha}}}}} M\bk{w,\lambda}
	\end{equation}
	(to be called \emph{normalized intertwining operators}) satisfy
	the following cocycle condition:
	\begin{equation}
	\label{Eq_cocycle_condition}
	\forall w,w'\in W:N\bk{ww',\lambda} = N\bk{w,w'\cdot\lambda}\circ N\bk{w',\lambda}\,.
	\end{equation}
	By construction we clearly have:
	\begin{equation}
	N\bk{w,\lambda} f_\lambda^0 = f_{w\cdot\lambda}^0 .
	\end{equation}
	
	\item \textit{(Induction in stages)}
	Given a simple reflection $w_\alpha$, $N\bk{w_\alpha,\lambda}$ factors through induction in stages.
	Namely, given the embedding $\iota_\alpha:SL_2\bk{F}\to G$ associated to the simple root $\alpha$, the following diagram is commutative:
	\begin{equation}
	\label{Eq_induction_in_stages}
	\xymatrix@C=6em{
	\Ind_B^G\lambda \ar@{->}[r]^{N\bk{w_\alpha,\lambda}} \ar@{->}[d]_{\iota_\alpha^\ast} &
	\Ind_B^G\bk{w_\alpha\cdot\lambda} \ar@{->}[d]^{\iota_\alpha^\ast} \\
	\Ind_{\mathcal{B}}^{SL_2\bk{F}} \bk{\gen{\lambda,\check{\alpha}}} \ar@{->}[r]^{N\bk{w_{\square},\gen{\lambda,\check{\alpha}}}} &
	\Ind_{\mathcal{B}}^{SL_2\bk{F}} \bk{\gen{w_{\square}\cdot \lambda,\check{\alpha}}}
	} ,
	\end{equation}	
	where $\mathcal{B}$ is the standard Borel subgroup $SL_2\bk{F}$, $w_{\square}=\begin{pmatrix}1&0\\0&-1\end{pmatrix}$ is the non-trivial Weyl element of $SL_2\bk{F}$ and the vertical maps in the diagram should be understood as the pull-back map.
	
	\item \textit{(Representations of $SL_2\bk{F}$)}
	We consider the representation $\pi_{s}=\Ind_{\mathcal{B}}^{SL_2\bk{F}} \FNorm{\omega}^s$, where $\omega$ is the unique fundamental weight on the torus of $SL_2\bk{F}$.
	The representation $\pi_{s}$ is irreducible for $s\neq \pm 1$.
	For $s=\pm 1$ we have the following exact sequences
	\begin{equation}
	\begin{array}{l}
	0\longrightarrow \Id_{SL_2\bk{F}} \longrightarrow \Ind_{\mathcal{B}}^{SL_2\bk{F}} \FNorm{\omega}^{-1} \longrightarrow \St_{SL_2\bk{F}} \longrightarrow 0 \\
	0\longrightarrow \St_{SL_2\bk{F}} \longrightarrow \Ind_{\mathcal{B}}^{SL_2\bk{F}} \FNorm{\omega}^{+1} \longrightarrow \Id_{SL_2\bk{F}} \longrightarrow 0 ,
	\end{array}
	\end{equation}
	where $\St_{SL_2\bk{F}}$ denotes the Steinberg representations of $SL_2\bk{F}$.
	Note that these sequences do not split.
	
	Furthermore, writing the Laurent series of $N\bk{w_{\square},s}$ around $s=-1$ and $s=+1$ yields
	\begin{equation}
	\label{Eq_Laurent_series_intertwining_operators}
	\begin{array}{l}
	N\bk{w_{\square},s} = \sum_{i=0}^\infty \bk{s+1}^i \mathcal{A}_i \\
	N\bk{w_{\square},s} = \sum_{i=-1}^\infty \bk{s-1}^i \mathcal{C}_i,
	\end{array}
	\end{equation}
	where 
	\begin{equation}
	\label{Eq_coefficients_Laurent_series_intertwining_operators}
	\begin{array}{l}
	Im\bk{\mathcal{A}_0} = \Id_{SL_2\bk{F}}, \quad
	Ker\bk{\mathcal{A}_0} = \St_{SL_2\bk{F}} \\
	Im\bk{\mathcal{C}_{-1}} = \St_{SL_2\bk{F}}, \quad
	Ker\bk{\mathcal{C}_{-1}} = \Id_{SL_2\bk{F}} .
	\end{array}
	\end{equation}
\end{itemize}

\subsection{The Langlands Subrepresentation Theorem}

We recall here the Langlands subrepresentation theorem.
See \cite[Chapter IV, Sec. XI.2]{MR1721403}, \cite{MR2050093} or \cite{MR2490651} for more details.
Note that most sources describe the quotient version of the Langlands classification theorem rather than the subrepresentation version we use here.
By taking contragredients, the two versions are equivalent.

Let $Q$ be a standard parabolic subgroup of $G$ with Levi subgroup $L$.
Let 
\[
\mathfrak{a}_L^+=\set{\lambda\in\mathfrak{a}_{L,\R}^\ast \mvert \gen{\lambda,\check{\alpha}}<0\ \forall \alpha\in\Delta_L}.
\]
A representation $\sigma$  of $L$ is called \textbf{tempered} if $\sigma$ is a direct summand in a parabolic induction from a square-integrable representation.

A \textbf{standard module} is an induction $\Ind_Q^G \bk{\sigma\otimes\lambda}$, where $\sigma$ is a tempered representation of $L$ and $\lambda\in \mathfrak{a}_L^+$.

\begin{Thm}[\cite{MR2050093} Lem. 2.4]
	\label{Thm:Langlands_Subrepresentation_theorem}
	Let $\Ind_Q^G \bk{\sigma\otimes\lambda}$ be a standard module.
	Then $\Ind_Q^G \bk{\sigma\otimes\lambda}$ admits a unique irreducible subrepresentation $\tau$ and $\tau$ is the kernel of $N\bk{w_L,\lambda}$, where $w_L$ is the shortest representative in $W$ of the class of the longest element in $W_L\lmod W$.
\end{Thm}

The operator $N\bk{w_L,\lambda}$ is called the \textbf{Langlands operator} for the standard module $\Ind_Q^G \bk{\sigma\otimes\lambda}$.

We note the following useful corollary of \Cref{Thm:Langlands_Subrepresentation_theorem}.

\begin{Cor}
	\label{Cor:Induction_from_antidominant_weight_has_a_UIS}
	Let $\lambda\in \mathfrak{a}_{T,\R}^\ast$ be anti-dominant, in the sense that $\gen{\lambda,\check{\alpha}} \leq 0$ for all $\alpha\in\Delta$.
	Then $\Ind_B^G\lambda$ admits a unique irreducible subrepresentation.
\end{Cor}

In order to prove \Cref{Cor:Induction_from_antidominant_weight_has_a_UIS}, we need the following fact:

\begin{Lem}
	\label{Appendix:Induction_from_trivial_is_Irreducible}
	The representation $\pi=Ind_B^G \Id_T$ is irreducible.
\end{Lem}

\begin{proof}
	We follow the ideas of \cite{MR620252,MR644669,MR582703,MR0460543}.
	Harish-Chandra's commuting algebra theorem states that the algebra $End_G\bk{\pi}$ is generated by $N\bk{w,\Id_M,0}$ where $w\in\Stab_W\bk{\Id_T}=W$.
	However, a simple calculation shows that $N\bk{w,\Id_M,0}=Id$ for any $w\in\Stab_W\bk{\Id_T}$ and hence $End_G\bk{\pi}\cong \C$.
	
	On the other hand, $\pi$ is unitary of finite length and hence isomorphic to a direct sum of irreducible representation $\oplus_{i=1}^l \sigma_i$.
	It follows that $\dim\bk{End_G\bk{\pi}}\geq l$.
	Hence $l=1$ and $\pi$ is irreducible.
\end{proof}

\begin{proof}[Proof of \Cref{Cor:Induction_from_antidominant_weight_has_a_UIS}]
	Let
	\[
	\Delta_L = \set{\alpha\in\Delta \mvert \gen{\lambda,\check{\alpha}}=0},
	\]
	$P=P_{\Delta_L}$ and let $L=M_{\Delta_L}$ be the (maximal) standard Levi subgroup such that the restriction of $\lambda$ to $L^{der}$ is trivial.
	
	By \Cref{Appendix:Induction_from_trivial_is_Irreducible}, $\Ind_{B\cap L}^L \lambda$ is an irreducible representation and can, in fact, be written as $\sigma\otimes\lambda'$, where $\sigma$ is a tempered representation of $L$ and $\lambda\in \mathfrak{a}_L^+$.
	\Cref{Cor:Induction_from_antidominant_weight_has_a_UIS} then follows from \Cref{Thm:Langlands_Subrepresentation_theorem}.
\end{proof}




%
%
%
%

\section{Decomposition with Respect to Levi Subgroups of Semi-Simple Rank 1}
\label{Section:Levi_Subgroups_of_SS_rank_1}

In this section we prove our main result of this paper, \Cref{Thm_Main}.
Before stating and proving it, we start by setting up some notations and listing the assumptions of this theorem.
While this list of assumptions may seem incomprehensible at first glance, in \Cref{Sec:Examples} we prove the existence of points $\lambda_0\in\mathfrak{a}_{T,\C}^\ast$ such that $\Ind_B^G\lambda_0$ decompose as in \Cref{Thm_Main}.
In fact, we show that if $rank(G)>2$, then there are infinitely many such points $\lambda_0$.

\begin{enumerate} [(\bfseries {Assumption} 1)]
\itshape

\item Fix a simple root $\alpha\in\Delta$.
\label{Assumption:SimpleRootEnumeration}
\setcounter{qcounter}{\value{enumi}}
\end{enumerate}

We make the following notations:
\begin{itemize}
\item Let $H_1=\set{\lambda\in \mathfrak{a}_\C^\ast \mvert \gen{\lambda,\check{\alpha}}=1}$, this is a hyperplane in $\mathfrak{a}_{T,\C}^\ast$.
\item Let $P=P_{\set{\alpha}}$ and $M=M_{\set{\alpha}}$.
%

\item
Let $A_M$ denote the central torus of $M$ and $M^{der}=\coset{M,M}$ be the derived group of $M$.
We have $A_M\subset T$ and hence $\mathfrak{a}_{M,\C}^\ast \hookrightarrow \mathfrak{a}_{T,\C}^\ast$.
In fact, the image of this embedding can be identified as those elements $\lambda\in\mathfrak{a}_{T,\C}^\ast$ satisfying $\gen{\lambda,\check{\alpha}}=0$.
Any character of $M$ is a trivial extension of a character of $A_M$.
Namely, of the form $\chi\boxtimes \Id_{M^{der}}$, where $\chi$ is a character of $A_M$, trivial on $A_M\cap M^{der}$.
Under these notations, it holds that $$\chi_0=\bk{\lambda_0-\frac{\alpha}{2}}\res{A_M}\boxtimes\Id_L .$$
Alternatively, $\chi_0$ is a character of $M$ such that
\[
\chi_0\res{T} = \lambda_0 - \rho_M = \lambda_0 - \frac{\alpha}{2} .
\]
\end{itemize}

\begin{enumerate}[(\bfseries {Assumption} 1)]
\itshape
\setcounter{enumi}{\value{qcounter}}
\item Fix $\lambda_0\in H_1$ such that $\Stab_W\bk{w_\alpha\cdot\lambda_0}\neq\set{1}$.
We note that $w_\alpha\notin \Stab_W\bk{w_\alpha\cdot\lambda_0}$.
\label{Assumption:Stabilizer_is_not_trivial}

\item Fix $1\neq w_0 \in \Stab_W\bk{w_\alpha\cdot\lambda_0}$ and assume that $N\bk{w_0,w_\alpha\cdot\lambda_0}=Id$.
\label{Assumption:w2ActsAsId}

\setcounter{qcounter}{\value{enumi}}
\end{enumerate}


We denote
\begin{equation}
\begin{split}
H_{-1} & = \set{\lambda\in\mathfrak{a}_\R^\ast \mvert \gen{w_0w_\alpha\cdot\lambda,\check{\alpha}}=-1} \\
& = \set{\lambda\in\mathfrak{a}_\R^\ast \mvert \gen{w_\alpha w_0w_\alpha\cdot\lambda,\check{\alpha}}=1}.
\end{split}
\end{equation}
Note that $\lambda_0 \in H_1\cap H_{-1}$.

\begin{enumerate}[(\bfseries {Assumption} 1)]
\itshape
\setcounter{enumi}{\value{qcounter}}

\item Assume that $H_1\neq H_{-1}$.
Equivalently, assume that $w_0$ does not commute with $w_\alpha$ (see \Cref{App_Lem_1} and \Cref{App_Lem_3} in \Cref{Appendix:Facts_on_roots_systems_and_Weyl_groups}).
\label{Assumption:H1H-1NotEqual}


\item Fix a line $\restline\subset \mathfrak{a}_\R^\ast$ such that $\restline\cap H_1 =\restline\cap H_{-1}=\set{\lambda_0}$ and that the angle between $\restline$ and $H_1$ is not supplementary to the angle between $\restline$ and $H_{-1}$.
\label{Assumption:LineS}

\setcounter{qcounter}{\value{enumi}}
\end{enumerate}

The existence of such a line $\restline$ follows from ({\bfseries {Assumption} \ref{Assumption:H1H-1NotEqual}}).
Namely, $H_1$ and $H_{-1}$ are distinct (affine) hyperplanes and hence of dimension $n-1$, and hence their intersection has (at most) dimension $n-2$.

\begin{enumerate}[(\bfseries {Assumption} 1)]
\itshape
\setcounter{enumi}{\value{qcounter}}

\item Assume that $\Ind_{P}^G \tr_M\otimes \chi_0$ admits a unique irreducible subrepresentation.
\label{Assumption:UISTrivial}

\item Assume that $\Ind_{P}^G \St_M\otimes \chi_0$ admits a unique irreducible subrepresentation.
\label{Assumption:UISSteinberg}
\setcounter{qcounter}{\value{enumi}}
\end{enumerate}

\begin{Thm}
\label{Thm_Main}
Assuming that the data $\bk{\lambda_0,\alpha,w_0}\in \mathfrak{a}_{T,\R}^\ast\times\Delta\times\Stab_W\bk{\lambda_0}$ satisfy assumptions \textbf{\ref{Assumption:SimpleRootEnumeration}-\ref{Assumption:UISTrivial}}.
Then
\begin{equation}
\label{Eq:Thm_Main}
\Ind_B^G {\lambda_0} = 
\bk{\Ind_{P}^G \St_M\otimes \chi_0} \oplus 
\bk{\Ind_{P}^G \tr_M\otimes \chi_0} .
\end{equation}

Furthermore, assuming \textnormal{({\bfseries {Assumption} \ref{Assumption:UISSteinberg}})}, each of $\Ind_{P}^G \St_M\otimes \chi_0$ and $\Ind_{P}^G \tr_M\otimes \chi_0$ admits a unique irreducible subrepresentation and the maximal semi-simple subrepresentation of $\Ind_B^G {\lambda_0}$ is of length $2$.

\end{Thm}

\begin{proof}
In order to prove the theorem, we compute the limit of $N\bk{w_\alpha w_0w_\alpha,\lambda}$ at $\lambda_0$ \textbf{along} the line $\restline$ and show that the direct summands in \Cref{Eq:Thm_Main} are both eigenspaces of that operator.

Let $\ell_1,...,\ell_n:\mathfrak{a}_{T,\R}^\ast\to\C$  denote a set of affine functions such that:
\begin{enumerate}
\item $\ell_i\bk{\lambda_0}=0$ for all $1\leq i\leq n$. In particular, $\ell\bk{\lambda-\lambda_0}$ is a linear functional on $\mathfrak{a}_{T,\R}^\ast$.
\item $\set{\nabla \ell_i\mvert 1\leq i\leq n}$ forms an orthogonal system in $\mathfrak{a}_{T,\R}^\ast$.
\item $\ell_1\bk{\lambda}=\gen{\lambda,\check{\alpha}}-1$.
\item $\ell_2\bk{\lambda}=\gen{w_0w_\alpha\cdot\lambda,\check{\alpha}}+1$.
\end{enumerate}
This can be done due to ({\bfseries {Assumption} \ref{Assumption:H1H-1NotEqual}}).

Note that any meromorphic function $\varphi$ in the neighbourhood of $\lambda_0$ has a Laurent expansion of the form
\[
\varphi\bk{\lambda} = \suml_{\vec{k}\in\Z^n} \bk{\prodl_{i=1}^n \ell_i\bk{\lambda}^{k_i}} \varphi_{\vec{k}} 
\]
with $\varphi_{\vec{k}}$ in the range of $\varphi$ (in what follows, we consider operator-valued meromorphic functions).

We start by writing the Laurent expansions of some normalized standard intertwining operators in the neighborhood of $\lambda_0$:
\begin{align*}
& N\bk{w_\alpha,\lambda}= \suml_{i=0}^\infty \bk{\gen{\lambda,\check{\alpha}}-1}^i A_i = \suml_{i=0}^\infty \ell_1\bk{\lambda}_i A_i, \\
& N\bk{w_0,w_\alpha\cdot\lambda} = Id + \suml_{\vec{k}\in\N^n} \bk{\prodl_{i=1}^n \ell_i\bk{\lambda}^{k_i}} B_{\vec{k}}, \\
& N\bk{w_\alpha,\bk{w_0w_\alpha}\cdot \lambda} = \suml_{i=-1}^\infty \bk{\gen{\lambda,\bk{w_0w_\alpha}^{-1}\cdot\check{\alpha}}+1}^i C_i = \suml_{i=-1}^\infty \ell_2\bk{\lambda}^i C_i .
\end{align*}
Here
\begin{align*}
& A_i\in \Hom_\C\bk{\Ind_B^G{\lambda_0},\Ind_B^G\bk{w_\alpha\cdot\lambda_0}}, \\
& B_{\vec{k}} \in \operatorname{End}_\C\bk{\Ind_B^G\bk{w_\alpha\cdot{\lambda_0}}}, \\
& C_i\in \Hom_\C\bk{\Ind_B^G{w_\alpha\cdot\lambda_0},\Ind_B^G\bk{\lambda_0}} .
\end{align*}
Note that $A_0$, $B_{\vec{0}}$ and $C_{-1}$ are $G$-equivariant but the rest of the operators $A_i$, $B_{\vec{k}}$ and $C_i$ need not be $G$-equivariant.
We further note that, by ({\bfseries {Assumption} \ref{Assumption:w2ActsAsId}}), $B_{\vec{0}}=Id$.
On the other hand, by \Cref{Eq_induction_in_stages} and \Cref{Eq_coefficients_Laurent_series_intertwining_operators}:
\[
\begin{split}
& Im\bk{A_0} = \Ind_{P}^G\bk{\tr_M\otimes \chi_0},\quad 
Ker\bk{A_0} =  \Ind_{P}^G\bk{\St_M\otimes\chi_0} \\
& Im\bk{C_{-1}} = \Ind_{P}^G\bk{\St_M\otimes\chi_0}, \quad
Ker\bk{C_{-1}} =  \Ind_{P}^G\bk{\tr_M\otimes \chi_0} .
\end{split}
\]

It follows, from \Cref{Eq_cocycle_condition}, that
\[
\begin{split}
& N\bk{w_\alpha,\bk{w_\alpha w_0w_\alpha}\cdot\lambda}\circ N\bk{w_\alpha,\bk{w_\alpha w_0}\cdot\lambda} = Id \\
& N\bk{w_\alpha, \bk{w_\alpha w_0}\cdot\lambda}\circ N\bk{w_\alpha, \bk{w_\alpha w_0w_\alpha}\cdot\lambda} = Id
\end{split}
\]
for any $\lambda\in \mathfrak{a}^\ast_\C$.
By evaluating the leading terms of both the left-hand side and right-hand side of these equations, we conclude that
\begin{equation}
\label{Eq:IdentitesOfCoefficients}
\begin{split}
& C_{-1}A_0 = 0 = A_0C_{-1} \\
& C_0A_0-C_{-1}A_1 = Id = A_0C_0-A_1C_{-1} .
\end{split}
\end{equation}

Note that
\begin{align*}
& N\bk{w_\alpha w_0w_\alpha,\lambda}
 = N\bk{w_\alpha,w_0w_\alpha\cdot\lambda} \circ N\bk{w_0,w_\alpha\cdot\lambda}\circ N\bk{w_\alpha,\lambda} \\
& = 
\coset{\suml_{i=-1}^\infty \ell_2\bk{\lambda}^i C_i} \circ
\coset{Id + \suml_{\vec{k}\in\N^n} \bk{\prodl_{i=1}^n \ell_i\bk{\lambda}^{k_i}} B_{\vec{k}}} \circ
\coset{\suml_{i=0}^\infty \ell_1\bk{\lambda}_i A_i} \\
& = \frac{1}{\ell_2\bk{\lambda}} C_{-1}A_0 + \frac{\ell_1\bk{\lambda}}{\ell_2\bk{\lambda}}C_{-1}A_1 + C_{-1}\bk{\frac{\suml_{i=1}^n \ell_i\bk{\lambda}}{\ell_2\bk{\lambda}}B_{\hat{e_i}}} A_0 + C_0A_0 + \suml_{\stackrel{\vec{k}\in\N^n}{\FNorm{\vec{k}}\geq 1}} \bk{\prodl_{i=1}^n \ell_i\bk{\lambda}^{k_i}} N_{\vec{k}}
\end{align*}
is a
Laurent series for $N\bk{w_\alpha w_0w_\alpha,\lambda}$ in a neighbourhood of $\lambda_0$, where:
\begin{itemize}
\item $\hat{e_i}=\bk{\delta_{i,j}}_{j=1,...,n}$ are the standard basis vectors in $\R^n$.
\item For $\vec{k}\in\Z^n$, we write $\FNorm{\vec{k}}=\suml_{i=1}^n \FNorm{k_i}$.
\item $N_{\vec{k}}$ is the corresponding $\vec{k}$-coefficient in the Laurent series of $N\bk{w_\alpha w_0w_\alpha,\lambda}$; when $\FNorm{\vec{k}}\geq 1$, these coefficients will not play a role in the following computations.
\end{itemize}

Restricting $N\bk{w_\alpha w_0w_\alpha,\lambda}$ (in the $\lambda$ variable) to $\restline$ yields:
\begin{align*}
N\bk{w_\alpha w_0w_\alpha,\lambda}\res{\restline}
& = 
+ \frac{\ell_1\bk{\lambda}}{\ell_2\bk{\lambda}}C_{-1}A_1
+ C_{-1} \bk{\suml_{i=1}^n \frac{\ell_i\bk{\lambda}}{\ell_2\bk{\lambda}} B_{\hat{e_i}} } A_0
+ C_0A_0 \\
& = \frac{\ell_1\bk{\lambda}}{\ell_2\bk{\lambda}}C_{-1}A_1
+ C_{-1} \bk{\suml_{i=1}^n \frac{\ell_i\bk{\lambda}}{\ell_2\bk{\lambda}} B_{\hat{e_i}} } A_0
+ C_0A_0.
\end{align*}

For a vector $v\neq 0$, parallel to $\restline$, we define
\begin{equation}
\kappa_i = \lim\limits_{\lambda\to\lambda_0} \coset{\frac{\ell_i\bk{\lambda}}{\ell_2\bk{\lambda}}\res{\restline}} =
\frac{\gen{\nabla \ell_i,v}}{\gen{\nabla \ell_2,v}} .
\end{equation}
The fact that this limits indeed exist, i.e. $\gen{\nabla \ell_2,v}\neq 0$, is due to ({\bfseries {Assumption} \ref{Assumption:LineS}}).
Note that $\kappa_i$ is independent of the choice of $v$.
Taking the limit of $N\bk{w_\alpha w_0w_\alpha,\lambda}$ at $\lambda_0$ along $\restline$ yields
\begin{equation}
\begin{split}
E
& = \lim\limits_{\lambda\to\lambda_0} \coset{N\bk{w_\alpha w_0w_\alpha,\lambda}\res{\restline}} \\
& = \kappa_1C_{-1}A_1 + C_{-1} B A_0 + C_0A_0 = -\kappa_1 Id + \bk{\kappa_1+1}C_0A_0 + C_{-1}BA_0 ,
\end{split}
\end{equation}
where
\[
B=\suml_{i=1}^n \kappa_i B_{\hat{e_i}} .
\]

We note that $E\in \operatorname{End}_G\bk{\Ind_B^G\lambda_0}$.
Define
\begin{equation}
\label{Eq_definition_of _projection}
P = \frac{1}{1+\kappa_1}\bk{Id-E} \in \operatorname{End}_G\bk{\Ind_B^G\lambda_0}.
\end{equation}
This is well defined, i.e. $\kappa_1\neq -1$, due to ({\bfseries {Assumption} \ref{Assumption:LineS}}).

\underline{Claim:} $P$ is a projection.

Indeed, applying \Cref{Eq:IdentitesOfCoefficients},
\begin{align*}
E^2 
= & \kappa_1^2 Id - 2\kappa_1\bk{\kappa_1+1}C_0A_0 - 2\kappa_1C_{-1}BA_0 + \bk{\kappa_1+1}^2C_0A_0C_0A_0 \\
& + \bk{\kappa_1+1}C_0A_0C_{-1}BA_0 + \bk{\kappa_1+1}C_{-1}BA_0C_0A_0 + C_{-1}BA_0C_{-1}BA_0 \\
= & \kappa_1^2 Id - 2\kappa_1\bk{\kappa_1+1}C_0A_0 - 2\kappa_1C_{-1}BA_0 \\
& + \bk{\kappa_1+1}^2C_0A_0\bk{Id+C_{-1}A_1} + \bk{\kappa_1+1}C_{-1}BA_0\bk{Id+C_{-1}A_1} \\
= & \kappa_1^2 Id + \bk{1-\kappa_1^2} C_0A_0 + \bk{1-\kappa_1}C_{-1}BA_0  = \kappa_1 Id + \bk{1-\kappa_1}E
\end{align*}
and hence
\[
P^2 = \frac{1}{\bk{1+\kappa_1}^2} \bk{Id-E^2} = \frac{1}{\bk{1+\kappa_1}^2} \bk{Id-2E+\kappa_1 Id+\bk{1-\kappa_1}E} = \frac{1}{1+\kappa_1} \bk{Id-E} = P .
\]

Since $P$ is a $G$-equivariant and a projection, it follows that
\begin{equation}
\Ind_B^G\lambda_0 = \Image P \oplus \Ker P .
\end{equation}

It remains to prove that $Ker A_0 = Im P$ and $Im A_0 = Ker P$.

Since 
\[
P = Id - \bk{C_0 - \frac{1}{\kappa_1+1}C_{-1}B } A_0
\]
it follows that $Ker A_0\subseteq Im P$.
Assume the $Ker A_0\subsetneq Im P$.
Note that $Id-P$ is a projection on $Ker P$.
It holds that
\begin{equation}
\label{Eq:DecompositionofA0}
A_0 = A_0\circ P + A_0\circ\bk{Id-P} .
\end{equation}
By our assumption $A_0\circ P\neq 0$.
We note that, since $Ev^0=v^0$, $v^0\in Ker P$ and hence $A_0\circ\bk{Id-P}\neq 0$.
It follows that $Im A_0$ has at least two irreducible subrepresentations in contradiction with the fact that, by ({\bfseries {Assumption} \ref{Assumption:UISTrivial}}), it has a unique irreducible subrepresentation.
We conclude that $Ker A_0 = Im P$ and, from \Cref{Eq:DecompositionofA0}, it follows that $Im A_0 = Ker P$.


\end{proof}

\begin{Remark}
It follows from the proof that $\Ind_{P}^G\bk{\tr_M\otimes \chi_0}$ is the eigenspace of $E$ of eigenvalue $1$ and $\Ind_{P}^G\bk{\St_M\otimes\chi_0}$ is eigenspace of eigenvalue $-\kappa_1\neq 1$.
We note here that the decomposition in \Cref{Eq:Thm_Main} and the projection $P$ in \Cref{Eq_definition_of _projection} are independent of $\restline$ and only the eigenvalues of $E$ depend on $\restline$.
\end{Remark}

%

Using induction in stages, \Cref{Eq_induction_in_stages},  ({\bfseries {Assumption} \ref{Assumption:UISTrivial}}) may be replaced with the following weaker assumption:
\begin{enumerate}[(\bfseries {Assumption} 1)]
\itshape
\item[(\bfseries {Assumption} \ref{Assumption:UISTrivial}')]
Let $L$ be a standard Levi containing $w_\alpha$ and $w'$ (and hence $M\subset L$) and assume that $\Ind_{P\cap L}^L \tr_M\otimes \chi_0$ admits a unique irreducible subrepresentation.
\end{enumerate}

\begin{Cor}
\label{Cor:SSRank1WithoutLanglandsPosition}
Under assumptions \textbf{\ref{Assumption:SimpleRootEnumeration}-\ref{Assumption:LineS}} and \textbf{\ref{Assumption:UISTrivial}'} \Cref{Eq:Thm_Main} holds.
\end{Cor}

\begin{proof}
Indeed, the conditions of \Cref{Thm_Main} applies to $\Ind_{B\cap L}^L\lambda_0$ and hence
\[
\Ind_{B\cap L}^L\lambda_0 = \Ind_{P\cap L}
\coset{\Ind_{P\cap L}^L \bk{\St_M\otimes \chi_0} \oplus \Ind_{P\cap L}^L \bk{\tr_M\otimes \chi_0}} .
\]
Applying induction by stages yield
\begin{align*}
\Ind_B^G {\lambda_0} 
& = \Ind_Q^G \bk{\Ind_{B\cap L}^L\lambda_0} \\
& = \Ind_Q^G\bk{\Ind_{{P}\cap L}^L \St_M\otimes \chi_0} \oplus \Ind_Q^G\bk{\Ind_{{P}\cap L}^L \tr_M\otimes \chi_0} \\
& = \Ind_{P}^G \bk{\St_M\otimes \chi_0} \oplus \Ind_{P}^G \bk{\tr_M\otimes \chi_0} ,
\end{align*}
where $Q$ is the standard parabolic subgroup whose Levi subgroup is $L$.
\end{proof}



\section{Existence of $\lambda_0$}
\label{Sec:Examples}
One question which arises from the discussion in \Cref{Section:Levi_Subgroups_of_SS_rank_1} is whether there exist points $\lambda_0$ which satisfy the assumptions of \Cref{Thm_Main}.
In this section, we show that for any simple group $G$ (satisfying the assumptions in \Cref{Section:Preliminaries}) and any simple root of $G$, one can choose $\lambda_0$ as in \Cref{Thm_Main}.
We prove:
\begin{Thm}
	\label{Thm:Key_Example}
	Fix $\alpha\in\Delta$, $\lambda'\in \mathfrak{a}_\R^\ast$ and $S\subset \Delta\setminus\set{\alpha}$ satisfying:
	\begin{enumerate}
		\item There exists $\beta\in S$ such that $\gen{\beta,\check{\alpha}}\neq 0$ (i.e. $\alpha$ and $\beta$ are neighbours in the Dynkin diagram of $G$).
		\item $\gen{\lambda',\check{\alpha}}=-1$.
		\item $\gen{\lambda',\check{\beta}}=0\quad \forall \beta\in S$.
		\item $\gen{\lambda',\check{\beta}} < 0\quad \forall \beta\notin S\cup\set{\alpha}$.
	\end{enumerate}
	Then, for $\lambda=w_\alpha\cdot\lambda'$ amd $M=M_{\set{\alpha}}$, it holds that
	\[
	\Ind_B^G {\lambda_0} = 
	\bk{\Ind_{P}^G \St_M\otimes \chi_0} \oplus 
	\bk{\Ind_{P}^G \tr_M\otimes \chi_0} .
	\]
	Furthermore, both $\Ind_{P}^G \tr_M\otimes \chi_0$ admits a unique irreducible subrepresentation and if $\St_M$ is irreducible, then so does $\Ind_{P}^G \St_M\otimes \chi_0$.
\end{Thm}

\begin{Remark}
	Note that the set of $\lambda_0$ satisfying the conditions in \Cref{Thm:Key_Example} has dimension $n-2$ and it is non-empty.
\end{Remark}

\begin{Remark}
	We note here that the Steinberg representation of $SL_2\bk{\R}$ has length $2$.
\end{Remark}

By choosing $S=\Delta\setminus\set{\alpha}$ in \Cref{Thm:Key_Example} we have:
\begin{Cor}
	\label{Cor:Key_Example}
	For any group $G$, as in \Cref{Section:Preliminaries}, and any simple root $\alpha\in\Delta$, let $\lambda_0=-w_\alpha\cdot\omega_\alpha$. 
	Then
	\begin{equation}
	\Ind_B^G {\lambda_0} = 
	\bk{\Ind_{P}^G \St_M\otimes \chi_0} \oplus 
	\bk{\Ind_{P}^G \tr_M\otimes \chi_0} ,
	\end{equation}
	where $\chi_0$ is chosen as in \Cref{Section:Levi_Subgroups_of_SS_rank_1}.
\end{Cor}

\begin{Remark}
	The decompositions appearing in \cite[Sec. 8]{MW_AppendixIII} \cite[Lem. 3.1]{Zampera1997}, \cite[pg. 1260-1 CASE 1]{MR1426903}, \cite[Lem. 5.12]{LaoResidualSpectrum} and \cite[Subsec. 4.4]{SegalResiduesD4} are all special cases of \Cref{Cor:Key_Example}.
\end{Remark}

\begin{proof}[Proof of \Cref{Thm:Key_Example}]
	
In order to prove \Cref{Thm:Key_Example}, we construct a system of equalities and inequalities, \textbf{System I}, whose solutions are guaranteed to satisfy the assumptions of \Cref{Thm_Main}.
We then show that this system is equivalent to the system, \textbf{System IV}, given by the assumptions of \Cref{Thm:Key_Example}.
We list the assumptions of \Cref{Thm_Main} and reinterpret some of them as inequalities that will compose our system; other assumptions (i.e. ({\bfseries {Assumption} \ref{Assumption:SimpleRootEnumeration}}), ({\bfseries {Assumption} \ref{Assumption:Stabilizer_is_not_trivial}}) and ({\bfseries {Assumption} \ref{Assumption:H1H-1NotEqual}})) will be quoted verbatim in \textbf{System I}.
We drop ({\bfseries {Assumption} \ref{Assumption:LineS}}) since, as explained in \Cref{Section:Levi_Subgroups_of_SS_rank_1}, it follows from ({\bfseries {Assumption} \ref{Assumption:H1H-1NotEqual}}). \\

\quad ({\bfseries {Assumption} \ref{Assumption:SimpleRootEnumeration}}) Fix a simple root $\alpha\in\Delta$. \\

\quad ({\bfseries {Assumption} \ref{Assumption:Stabilizer_is_not_trivial}}) Fix $\lambda_0\in H_1$ such that $\Stab_W\bk{w_\alpha\cdot\lambda_0}\neq\set{1}$. \\



\quad ({\bfseries {Assumption} \ref{Assumption:w2ActsAsId}}) Fix $1\neq w_0 \in \Stab_W\bk{w_\alpha\cdot\lambda_0}$ such that $N\bk{w_0,w_\alpha\cdot\lambda_0}=Id$. \\

Assume that $\lambda'=w_\alpha\cdot\lambda_0$ lies in the anti-dominant chamber.
Let $S=\set{\gamma\in\Delta\mvert \gen{\lambda',\check{\gamma}}=0}$.
It follows from ({\bfseries {Assumption} \ref{Assumption:Stabilizer_is_not_trivial}}) that $S\neq\emptyset$ and that $w_0\in W_{M_S}$.
By induction in stages, it holds that
\begin{equation}
\label{Eq_examples_induction_in_Stages}
\Ind_B^G \lambda' = \Ind_{P_S}^G \bk{\Ind_{B\cap {M_S}}^{M_S} \Id} \otimes \lambda' .
\end{equation}
We prove in \Cref{Appendix:Induction_from_trivial_is_Irreducible} that $\Ind_{B\cap {M_S}}^{M_S} \Id$ is irreducible.
It is also spherical and hence, by \Cref{Eq_induction_in_stages}, it follows that $N\bk{w_0,w_\alpha\cdot\lambda_0}=Id$. \\

\quad ({\bfseries {Assumption} \ref{Assumption:H1H-1NotEqual}}) Assume that $w_0$ does not commute with $w_\alpha$. \\




\quad ({\bfseries {Assumption} \ref{Assumption:UISTrivial}}) Assume that $\Ind_{P}^G \tr_M\otimes \chi_0$ admits a unique irreducible subrepresentation. \\

Let $\lambda'=w_\alpha\cdot \lambda_0$ and $S$ be as in \Cref{Eq_examples_induction_in_Stages}.
Since $\Ind_{P}^G \tr_M\otimes \chi_0$ embeds into $\Ind_B^G \lambda'$,
\Cref{Cor:Induction_from_antidominant_weight_has_a_UIS} implies ({\bfseries {Assumption} \ref{Assumption:UISTrivial}}). \\


\quad ({\bfseries {Assumption} \ref{Assumption:UISSteinberg}}) Assume that $\Ind_{P}^G \St_M\otimes \chi_0$ admits a unique irreducible subrepresentation. \\

If $\gen{\chi_0,\check{\beta}}\leq 0$ for any $\beta\in\Phi^{+}\setminus\set{\alpha}$ and $\St_M$ is irreducible, then ({\bfseries {Assumption} \ref{Assumption:UISSteinberg}}) follows from the Langlands' subrepresentation theorem since $\St_M$ is tempered (indeed, it is a discrete series representation) and $\Ind_{P}^G \St_M\otimes \chi_0$ is a standard module. \\

We summarize this discussion by the following system of equalities and inequalities:
\paragraph{\textbf{System I:}}
Pick $w\in W$ and $\lambda \in \mathfrak{a}_\R^\ast$ such that:
\begin{enumerate}
	\item $\coset{w,w_\alpha}\neq 1$.
	\item $ww_\alpha\cdot\lambda=w_\alpha\cdot\lambda$.
	\item $\gen{\lambda,\check{\alpha}}=1$.
	\item $\gen{\lambda-\frac{\alpha}{2},\check{\beta}}\leq 0\quad \forall \beta\in\Phi^{+}\setminus\set{\alpha}$.
	\item $\gen{w_\alpha\cdot\lambda,\check{\beta}}\leq 0\quad \forall \beta\in\Delta$.
\end{enumerate}

We now argue that \textbf{System I} is equivalent to the system in the statement of \Cref{Thm:Key_Example}, \textbf{System IV}.
We do this in stages by showing the equivalence of \textbf{System I}, \textbf{System II}, \textbf{System III} and \textbf{System IV}.

Note that $\gen{\lambda,\check{\alpha}}=1$ implies $w_\alpha\cdot\lambda=\lambda-\alpha$.
We make a change of variables $\lambda'=w_\alpha\cdot\lambda$ and get an equivalent system:
\paragraph{\textbf{System II:}}
Pick $w\in W$ and $\lambda' \in \mathfrak{a}_\R^\ast$ such that:
\begin{enumerate}
	\item $\coset{w,w_\alpha}\neq 1$.
	\item $w\lambda'=\lambda'$.
	\item $\gen{\lambda',\check{\alpha}}=-1$.
	\item $\gen{\lambda',\check{\beta}}\leq -\frac{1}{2} \gen{\alpha,\check{\beta}}\quad \forall \beta\in\Phi^{+}\setminus\set{\alpha}$.
	\item $\gen{\lambda',\check{\beta}}\leq 0\quad \forall \beta\in\Delta$.
\end{enumerate}

Since $\lambda'$ is anti-dominant, $Stab_W\bk{\lambda'}$ is generated by simple reflections.
In particular, $Stab_W\bk{\lambda'}=\gen{s_\beta\mvert \gen{\lambda',\check{\beta}}=0,\, \beta\in\Delta}$ is not trivial if and only if $\lambda'$ is on a wall of the chamber.


We now consider the following system:
\paragraph{\textbf{System III:}}
Pick a subset $S\subset \Delta\setminus\set{\alpha}$ and $\lambda' \in \mathfrak{a}_\R^\ast$ such that:
\begin{enumerate}
	\item There exist $\beta\in S$ such that $\gen{\beta,\check{\alpha}}\neq 0$.
	\item $\gen{\lambda',\check{\alpha}}=-1$.
	
	\item $\gen{\lambda',\check{\beta}}=0\quad \forall \beta\in S$.
	\item $\gen{\lambda',\check{\beta}} \leq 0\quad \forall \beta\notin S\cup\set{\alpha}$.
	
	\item $\gen{\lambda',\check{\beta}}\leq -\frac{1}{2}\gen{\alpha,\check{\beta}}\quad \forall \beta\in\Phi^{+}\setminus\set{\alpha}$.
\end{enumerate}

The set of solutions of this system equals the set of solutions of \textbf{System II} as will be explained now.

\begin{itemize}
	\item Let $w\in W$ and $\lambda' \in \mathfrak{a}_\R^\ast$ constitute a solution of \textbf{System II}.
	We automatically see that \textbf{II.3} implies \textbf{III.2},	\textbf{II.5} implies \textbf{III.4} and \textbf{II.4} implies \textbf{III.5}.
	Let
	\[
	S = \set{\beta \mvert \gen{\lambda',\check{\beta}}=0} .
	\]
	This choice automatically guarantees \textbf{System III.3}.
	It remains to show that \textbf{System III.1} holds.
	
	Assume that $\gen{\beta,\check{\alpha}}=0$ for all $\beta\in S$.
	\textbf{II.5} implies that $\lambda'$ is anti-dominant and hence $Stab_W\bk{\lambda'} = \gen{w_\beta\mvert \beta\in S}$.
	\textbf{II.2} implies that $Stab_W\bk{\lambda'}$ is non-trivial.
	In fact, it follows that $S\neq \emptyset$.
	If $\gen{\beta,\check{\alpha}}=0$ for all $\beta\in S$ it would imply that $\coset{w,w_\alpha}= 1$ for all $w\in Stab_W\bk{\lambda'}$ contradicting	\textbf{II.1}.
	
	\item
	Let $S\subset \Delta\setminus\set{\alpha}$ and $\lambda' \in \mathfrak{a}_\R^\ast$ constitute a solution of \textbf{System III}.
	We automatically see that \textbf{III.2} implies \textbf{II.3} and \textbf{III.5} implies \textbf{II.4}.
	Also,
	\textbf{III.2}, \textbf{III.3} and \textbf{III.4} implies \textbf{II.5} and, in particular, $\lambda'$ lies in the anti-dominant chamber.
	
	Again, $Stab_W\bk{\lambda'} = \gen{w_\beta\mvert \beta\in S}$ and \textbf{III.1} implies that there exists $w\in Stab_W\bk{\lambda'}$ such that $\coset{w,w_\alpha}\neq 1$ (say, $w=w_\beta$) so \textbf{II.1} and \textbf{II.2} hold.
	In particular, any solution of \textbf{System II} is attained this way.
\end{itemize}

It is shown in \Cref{Appendix:Facts_on_roots_systems_and_Weyl_groups} that, in fact, \textbf{III.5} is redundant.
Hence, \textbf{System III} is equivalent to the following system:
\paragraph{\textbf{System IV:}}
Pick a subset $S\subset \Delta\setminus\set{\alpha}$ and $\lambda' \in \mathfrak{a}_\R^\ast$ such that:
\begin{enumerate}
	\item There exist $\beta\in S$ such that $\gen{\beta,\check{\alpha}}\neq 0$.
	\item $\gen{\lambda',\check{\alpha}}=-1$.
	\item $\gen{\lambda',\check{\beta}}=0\quad \forall \beta\in S$.
	\item $\gen{\lambda',\check{\beta}} < 0\quad \forall \beta\notin S\cup\set{\alpha}$.
\end{enumerate}
\end{proof}

We now wish to consider a few particular examples of $G$, $\alpha$ and $\lambda_0$ given by \Cref{Thm:Key_Example}.
For simplicity, we assume that $F$ is non-Archimedean.

\begin{Ex}
	We consider simple, connected, simply-connected, split groups of rank $2$.
	In this case, $G$ is either of type $A_2$, $B_2=C_2$ or $G_2$.
	Namely, its Dynkin diagram is one of the following:
%
%
%
%
%
%

	\begin{center}
	\setlength{\tabcolsep}{12pt}
	\begin{tabular}{ccc}
	\begin{tikzpicture}
	\dynkin[open,label]{A}{2}
	\end{tikzpicture}
	&
	\begin{tikzpicture}
	\dynkin[open,label]{C}{2}
	\end{tikzpicture}
	&
	\begin{tikzpicture}
	\dynkin[open,label]{G}{2}
	\end{tikzpicture}
	\\
	Type $A_2$
	&
	Type $B_2=C_2$
	&
	Type $G_2$
	\end{tabular}
	\end{center}

	For each of these groups, and every $\alpha\in\Delta$, $S$ may be only $S=\Delta\setminus\set{\alpha}$.
	The possible $\lambda_0$ given by \Cref{Thm:Key_Example} are listed in the following table.
	\begin{center}
	\begin{tabular}{|c|c|c|}
		\hline
		& $\alpha_1$ & $\alpha_2$ \\ \hline
		$A_2$ & $\lambda_0=\bk{1,-1}$ & $\lambda_0=\bk{-1,1}$ \\ \hline
		$B_2=C_2$ & $\lambda_0=\bk{1,-1}$ & $\lambda_0=\bk{-2,1}$ \\ \hline
		$G_2$ & $\lambda_0=\bk{1,-1}$ & $\lambda_0=\bk{-3,1}$ \\ \hline
	\end{tabular}
	\end{center}

	For each of these points, we get a decomposition of the form
	\[
	\Ind_B^G\lambda_0 = \bk{\Ind_{P_{\set{\alpha_i}}}^G \Id_{M_{\set{\alpha_i}}} \otimes \chi_0} \oplus \bk{\Ind_{P_{\set{\alpha_i}}}^G \St_{M_{\set{\alpha_i}}}\chi_0},
	\]
	as in \Cref{Cor:Key_Example}.
	However, some of these points could be associated to a degenerate principal series representation induced from the other maximal parabolic.
	Namely, there exist an $s$ such that $I_{P_{\set{\alpha_{2-i}}}}\bk{s} = \Ind_{P_{\set{\alpha_{2-i}}}}^G\modf{P_{\set{\alpha_{2-i}}}}^s$ is a subrepresentation of $\Ind_B^G\lambda_0$.
	These degenerate principal series are given in the following table:
	\begin{center}
		\begin{tabular}{|c|c|c|}
			\hline
			& $\alpha_1$ & $\alpha_2$ \\ \hline
			$A_2$ & $I_{P_{\set{\alpha_{2}}}}\bk{\frac{1}{6}}$ & $I_{P_{\set{\alpha_{1}}}}\bk{\frac{1}{6}}$ \\ \hline
			$B_2=C_2$ & $I_{P_{\set{\alpha_{2}}}}\bk{0}$  & \\ \hline
			$G_2$ & $I_{P_{\set{\alpha_{2}}}}\bk{-\frac{1}{10}}$ & \\ \hline
		\end{tabular}
	\end{center}
	
	Let $\pi_1\oplus\pi_{-\kappa_1}$ be the maximal semi-simple subrepresentation of $\Ind_B^G\lambda_0$.
	Obviously, $\pi_1$ is a subrepresentation of $I_{P_{\set{\alpha_{2-i}}}}\bk{s}$.
	We wish do determine whether $\pi_{-\kappa_1}$ is also a subrepresentation of $I_{P_{\set{\alpha_{2-i}}}}\bk{s}$ or not.
	We answer this question for the $p$-adic case (in the Archimedean case the results are similar, while the arguments are more involved).

	\begin{itemize}
		\item \textbf{$A_2$ case}:
		In this case, $I_{P_{\set{\alpha_{1}}}}\bk{-\frac{1}{6}} = I_{P_{\set{\alpha_{2}}}}\bk{\frac{1}{6}}$ and $I_{P_{\set{\alpha_{1}}}}\bk{\frac{1}{6}} = I_{P_{\set{\alpha_{2}}}}\bk{-\frac{1}{6}}$ are both irreducible.
		Hence, $\pi_{-1}$ is not a subrepresentation of any of these degenerate principal series representations.
		
		\item \textbf{$B_2=C_2$ case}:
		This case was studied in \cite[pg. 9]{MR3267117}.
		In this case, we have $\pi_1\oplus\pi_{-1}$ as a subrepresentation of $I_{P_{\set{\alpha_{2}}}}\bk{0}$.
		In order to see this, one can compare the multiplicity of the exponent $\lambda_0$ in the Jacquet functor (along $N$) of $\Ind_B^G\lambda_0$, $\pi_1$, $\pi_{-1}$ and $I_{P_{\set{\alpha_{2}}}}\bk{0}$ (2, 1, 1 and 2 respectively).
		
		\item \textbf{$G_2$ case}:
		This case was studied in \cite[Lem. 3.1]{Zampera1997}.
		It is shown there that $\pi_{1}\oplus\pi_{-2}$ is a subrepresentation of $I_{P_{\set{\alpha_{2}}}}\bk{-\frac{1}{10}}$.
		This can be shown by comparing the multiplicity of the exponent $\lambda_0$ in the Jacquet functor (along $N$) of $\Ind_B^G\lambda_0$, $\pi_1$, $\pi_{-2}$ and $I_{P_{\set{\alpha_{2}}}}\bk{-\frac{1}{10}}$ (2, 1, 1 and 2 respectively).
		
	\end{itemize}	
	
\end{Ex}

In what follows, we use the following notations on the Dynkin-diagram:
\begin{itemize}
	\item We use $\bullet$ to denote the simple root $\alpha$.
	\item We use $\times$ to denote simple roots in $S$.
	\item We use $\circ$ to denote other simple roots.
	\item The $k$-vertex in a Dynkin diagram is associated to the simple root denoted $\alpha_k$.
	We further denote by $\omega_k$ the $k^{th}$ fundamental weight and by $w_k=w_{\alpha_k}$ the simple reflection associated to $\alpha_k$.
\end{itemize}

\begin{Ex}
Let $G=SL_4\bk{F}$ with the standard choice of $B$, $T$ and the enumeration of simple roots.
The group $G$ is of type $A_3$ and have the following Dynkin diagram:
\[
\begin{tikzpicture}
\dynkin[open,label]{A}{3}
\end{tikzpicture}.
\]

For $\alpha=\alpha_1$, we have two possible choices for the set $S$: either $\set{\alpha_2}$ or $\set{\alpha_2, \alpha_3}$.

For $\alpha=\alpha_2$, we have three possible choices for the set $S$: either $\set{\alpha_1}$, $\set{\alpha_3}$ or $\set{\alpha_1, \alpha_3}$.

The analysis for $\alpha=\alpha_3$ is similar to the case of $\alpha_1$.
\begin{center}
	\begin{tabular}{|>{\centering\arraybackslash}m{0.5cm}|>{\centering\arraybackslash}m{1cm}|>{\centering\arraybackslash}m{2.5cm}|>{\centering\arraybackslash}m{2cm}|>{\centering\arraybackslash}m{8cm}|}
		
		\hline
		$\alpha$ & & $S$ & $w$ & $\lambda_0$
		\\
		\hline
		
		$\alpha_1$ &
		\begin{tikzpicture}
		\dynkin[open,label,parabolic=4]{A}{3}
		\dynkincloseddot{1}
		\end{tikzpicture} &
		$S=\set{\alpha_2}$ & $w_2$ &
		$\lambda_0=-w_1\cdot\omega_1 - t\omega_3 = \bk{1,-1,-t} \quad\forall t>0$ \\ \hline
		
		$\alpha_1$ &
		\begin{tikzpicture}
		\dynkin[open,label,parabolic=12]{A}{3}
		\dynkincloseddot{1}
		\end{tikzpicture} &
		$S=\set{\alpha_2,\alpha_3}$ & $\begin{matrix} w_2, w_{23}, \\ w_{32}, w_{232}	\end{matrix}$ &
		$\lambda_0= -w_1\cdot\omega_1 = \bk{1,-1,0} $\\ \hline
		
		$\alpha_2$ &
		\begin{tikzpicture}
		\dynkin[open,label,parabolic=2]{A}{3}
		\dynkincloseddot{2}
		\end{tikzpicture} &
		$S=\set{\alpha_1}$ & $w_1$ &
		$\lambda_0= -w_2\cdot\omega_2-t\omega_3 =  \bk{-1,1,-t-1} \quad\forall t>0$ \\ \hline
		
		$\alpha_2$ &
		\begin{tikzpicture}
		\dynkin[open,label,parabolic=8]{A}{3}
		\dynkincloseddot{2}
		\end{tikzpicture} &
		$S=\set{\alpha_3}$ & $w_3$ &
		$\lambda_0= -t\omega_1-w_2\cdot\omega_2 =  \bk{-t-1,1,-1} \quad\forall t>0$ \\ \hline
		
		$\alpha_2$ &
		\begin{tikzpicture}
		\dynkin[open,label,parabolic=8]{A}{3}
		\dynkincloseddot{2}
		\end{tikzpicture} &
		$S=\set{\alpha_1, \alpha_3}$ & $w_1$, $w_3$, $w_{13}$ &
		$\lambda_0= -w_2\cdot\omega_2 = \bk{-1,1,-1}$ \\ \hline
	\end{tabular}
\end{center}
\end{Ex}

\begin{Ex}
	Another interesting example occurs in the case where $G$ is a quasi-split group of type $D_4$.
	The Dynkin diagram of the absolute root system of $G$, together with our choice of $\alpha$ and $S$, is given by
	
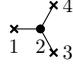
\[
	\begin{tikzpicture}
	\dynkin[open,label, parabolic=26]{D}{4}
	\dynkincloseddot{2}
	\end{tikzpicture}
	\]
	In this case, it follows, from \Cref{Thm:Key_Example}, that
	\[
	\Ind_B^G\lambda_0 = \bk{\Ind_{P_{\set{\alpha_2}}}^G \Id_{M_{\set{\alpha_2}}} \otimes \chi_0} \oplus \bk{\Ind_{P_{\set{\alpha_2}}}^G \St_{M_{\set{\alpha_2}}} \otimes \chi_0} ,
	\]
	where $\lambda_0 = -w_2\cdot \omega_2 = \bk{-1,1,-1,-1}$ and $\iota_{M_{\set{\alpha_2}}}\bk{\chi_0} = \bk{-\frac{1}{2},0,-\frac{1}{2},-\frac{1}{2}}$.
	In particular, let $\pi_1\oplus\pi_{-2}$ be the maximal semi-simple subrepresentation of $\Ind_B^G\lambda_0$.
	Note that the eigenvalue $-2$ is computed with respect to $w_0=w_1w_3w_4$ and $\restline=\bk{-1,s,-1,-1}$.
	
	As in the rank $2$ case, $\Ind_B^G\lambda_0$ contains a degenerate principal series representation.
	Namely, $\Ind_{P_{\set{\alpha_1,\alpha_3,\alpha_4}}}^G \modf{{P_{\set{\alpha_1,\alpha_3,\alpha_4}}}}^{-\frac{1}{10}}$ is a subrepresentation of $\Ind_B^G\lambda_0$.
	It is clear that $\pi_1$ is a subrepresentation of $\Ind_{P_{\set{\alpha_1,\alpha_3,\alpha_4}}}^G \modf{{P_{\set{\alpha_1,\alpha_3,\alpha_4}}}}^{-\frac{1}{10}}$ and the question is whether $\pi_{-2}$ is also a subrepresentation.
	This question was studied in detail in \cite[Subsec. 4.4]{SegalResiduesD4} and it is shown there that $\pi_{-2}$ is a subrepresentation of $\Ind_{P_{\set{\alpha_1,\alpha_3,\alpha_4}}}^G \modf{{P_{\set{\alpha_1,\alpha_3,\alpha_4}}}}^{-\frac{1}{10}}$ when the relative root system of $G$ is of type $G_2$ and that $\pi_1$ is the unique irreducible subrepresentation of $\Ind_{P_{\set{\alpha_1,\alpha_3,\alpha_4}}}^G \modf{{P_{\set{\alpha_1,\alpha_3,\alpha_4}}}}^{-\frac{1}{10}}$ when the relative root system of $G$ is of type $B_3$ or $D_4$.
\end{Ex}

\begin{Ex}
	As another example, let $G$ be the split, simply-connected, simple group of type $E_6$.
	The Dynkin diagram of $G$, together with our choice of $\alpha$ and $S$, is given by
	
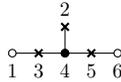
\[
	\begin{tikzpicture}
	\dynkin[open,label, parabolic=44]{E}{6}
	\dynkincloseddot{4}
	\end{tikzpicture}
	\]
	
	Let $\lambda'=\bk{-1,0,0,-1,0,-1}$ and $\lambda_0=w_\alpha\cdot\lambda'=\bk{-1,-1,-1,1,-1,-1}$.
	By \Cref{Thm:Key_Example}, it holds that
	\[
	\Ind_B^G {\lambda_0} = 
	\bk{\Ind_{P}^G \St_M\otimes \chi_0} \oplus 
	\bk{\Ind_{P}^G \tr_M\otimes \chi_0}
	\]
	and the maximal semi-simple subrepresentation of $\Ind_B^G {\lambda_0}$
	can be written as $\pi_1\oplus \pi_{-1}$.
	The degenerate principle series $\Pi=\Ind_{P_{\set{\alpha_1,\alpha_2,\alpha_3,\alpha_5,\alpha_6}}}^G \modf{{P_{\set{\alpha_1,\alpha_2,\alpha_3,\alpha_5,\alpha_6}}}}^{-\frac{3}{14}}$ is a subrepresentation of $\Ind_B^G {\lambda_0}$ so the maximal semi-simple subrepresentation of $\Pi$ is either $\pi_1$ or $\pi_1\oplus \pi_{-1}$.
	It is shown in \cite{HalawiSegal} that, in fact, $\pi_1$ is the unique irreducible subrepresentation of $\Pi$.
\end{Ex}

\section{Decomposition with Respect to Levi Subgroups of Higher Semi-Simple Rank}
\label{Section:Higher_rank_Levis}


In this section, we discuss a generalization of \Cref{Thm_Main}.
This generalization allows to consider points $\lambda_0$ where one could apply \Cref{Thm_Main} to triples $\bk{\lambda_0,\alpha,w_0^\alpha}$ with more than one simple root $\alpha$.
In such a case, one would be able to prove a finer decomposition of $\Ind_B^G \lambda_0$ into a direct sum of generalized degenerate principal series.

\subsection{Commuting Projections}

Let $\Theta=\set{\alpha_1,...,\alpha_k}\subset \Delta$, with $1\leq k\leq n$ and $\mathcal{P}\bk{k}=\set{X\subset\set{1,...,k}}$.
We recall the parabolic subgroup, $P_\Theta=\bigcap_{i=1}^k {P}_{\set{\alpha_i}}$, associated to $\Theta$.

For $X\in\mathcal{P}\bk{k}$, let $\St_X=\bk{\otimes_{i\in X} \tr_i} \otimes \bk{\otimes_{i\notin X}\St_i}$ where $\tr_i$ and $\St_i$ are the trivial and Steinberg representations of $M_i=M_{\set{\alpha_i}}^{der}$.

\begin{Cor}
	\label{Cor:Commuting_Projections}
	Assume that $\lambda_0$ satisfies assumptions 1-6 and 7' with respect to each triple $\bk{\lambda_0,\alpha_i,w_0^{\bk{i}}}$ for $\Theta=\set{\alpha_1,...,\alpha_k}\subset \Delta$.
	For each $1\leq i\leq k$, let $P_i$ be the projection on $\Ind_B^G \lambda_0$ constructed in \Cref{Eq_definition_of _projection} for $\bk{\lambda_0,\alpha_i,w_0^{\bk{i}}}$.
	Further assume that the projections $P_i$ are mutually commuting.
	Then
	\begin{equation}
	\label{Eq:Comm_Proj_Eq}
	\Ind_B^G\lambda_0 = \bigoplus_{X\in \mathcal{P}\bk{k}} \Ind _{P}^G \St_X \otimes \chi_0,
	\end{equation}
	where $\chi_0\in X^\ast\bk{M_\Theta}$ such that $\mathcal{J}_T^{M_\Theta}\chi_0=\lambda_0$.
	
\end{Cor}

\begin{Remark}
	\label{Rem:commuting_Weyl_elements_imply_commuting_projections}
	If $w_{\alpha_1} w_0^{\bk{1}}w_{\alpha_1}$, $w_2w_0^{\bk{2}}w_2$,..., $w_kw_0^{\bk{k}}w_k$ are all commuting, then so are $P_1$, $P_2$,..., $P_k$.
\end{Remark}

\begin{proof}
	
	For $X\in \mathcal{P}$, let
	\[
	P_X = \prodl_{i\in X} P_i \prodl_{i\notin X} \bk{Id-P_i} .
	\]
	One simply checks that $\set{P_X\mvert X\in\mathcal{P}\bk{k}}$ is a set of mutually orthogonal (and hence commuting) projections on $\Ind_B^G\lambda_0$ such that
	\begin{equation}
	\suml_{X\in \mathcal{P}\bk{k}} P_X = Id .
	\end{equation}
	It follows that
	\begin{equation}
	\Ind_B^G\lambda_0 = \bigoplus_{X\in \mathcal{P}\bk{k}} \Image\bk{P_X} .
	\end{equation}
	On the other hand, for $X\in \mathcal{P}\bk{k}$, we have
	\begin{equation}
	\begin{split}
	\Image\bk{P_X}
	& = \bigcap_{i\in X} \Image\bk{P_i} \cap \bigcap_{i\notin X} \Image\bk{Id-P_i} \\
	& = \bigcap_{i\in X} \bk{\Ind_{{P}_i}^{G}\Id_i\otimes\chi_0} \cap \bigcap_{i\notin X} \bk{\Ind_{{P}_i}^G \St_i\otimes\chi_0 } \\
	& = \Ind_{P}^G \St_X \otimes \chi_0.
	\end{split}
	\end{equation}
\end{proof}

\begin{Remark}
	If the projections $P_1,...,P_k$ were not commuting, one can show that the resulting enodmorphisms $P_X$ would be unipotent and not projective.
	This shows that some of the (not necessarily irreducible) constituents $\Ind_{P}^G \St_X \otimes \chi_0$ of $\Ind_B^G\lambda_0$ are not direct summands of $\Ind_B^G\lambda_0$.
\end{Remark}

\subsection{Examples}
We now wish to use \Cref{Thm:Key_Example}, \Cref{Cor:SSRank1WithoutLanglandsPosition} and \Cref{Rem:commuting_Weyl_elements_imply_commuting_projections} in order to find points $\lambda_0$ which satisfy the assumptions of \Cref{Cor:Commuting_Projections}.


For the sake of this computation, it is more convenient to consider triples $\bk{\lambda_0,\alpha_i,S_i}$, where $S_i\subset\Delta$ as in \Cref{Sec:Examples}, and let $w_0^{\bk{i}}\in W_S$ as in the proof of \Cref{Thm:Key_Example}.
 
In order to mark our choice of $\alpha_i$ and $S_i$ we use the following markings on the Dynkin diagram of $G$ (similar to the notations used in \Cref{Sec:Examples}):
\begin{itemize}
	\item We use $\bullet$ to denote the simple roots in $\Theta$.
	\item We use $\times$ to denote simple roots which lie in one of the $S_i$.
	\item We use $\circ$ to denote simple roots not in $\cup S_i$.
	\item The $k$-vertex in a Dynkin diagram is associated to the simple root denoted $\alpha_k$.
	We further denote by $\omega_k$ the $k^{th}$ fundamental weight and by $w_k=w_{\alpha_k}$ the simple reflection associated to $\alpha_k$.
\end{itemize}

We note that it is enough to consider root systems of type $A_n$, $D_n$ and $E_n$.
Since the underlying graph of type $B_n$, $C_n$, $G_2$ or $F_4$ is the same as that of $A_n$, it is enough to consider those.

Furthermore, in the following discussion, we make the following assumptions:
\begin{itemize}
	
	\item Consider the "horns", $\alpha_{n-1}$ and $\alpha_n$ of the Dynkin diagram of type $D_n$.
	Generically $w_{n-1}w_0^{\bk{n-1}}w_{n-1}$ and $w_{n}w_0^{\bk{n}}w_n$ will not commute and hence we do not treat this case.
	Hence, a "generic" choice of vertices on the Dynkin diagram of type $D_n$ can be done in the diagram of type $A_{n-1}$.
	
	It should be noted that, for particular choices of $w_0^{\bk{n-1}}$ and $w_0^{\bk{n}}$, these words do commute.
	
	\item For similar reasons, we consider only the cases where $\set{\alpha_i}\cup S_i$ are disjoint and for any $i\neq j$ the sub-Dynkin diagram with vertices $\set{\alpha_i,\alpha_j}\cup S_i \cup S_j$ is disjoint.
	In particular, we assume that $rank\bk{G}\geq 5$.
	
\end{itemize}

\begin{Ex}
Let $G$ be of type $A_5$ (i.e. $G=SL_6$).
There are three possible choices of $2$ vertices:
\begin{enumerate}
	\item Choosing the $1^{st}$ and $5^{th}$ vertices in the Dynkin diagram:
	\begin{center}
		\begin{tikzpicture}
		\dynkin[open,label,parabolic=20]{A}{5}
		\dynkincloseddot{1}
		\dynkincloseddot{5}
		\end{tikzpicture}
	\end{center}
	The possible associated points are $\lambda_0=-\bk{w_1\cdot\omega_1+w_5\cdot\omega_5}-t\omega_3$, where $t>0$.
	The two projections in \Cref{Cor:Commuting_Projections} are the ones associated to $N\bk{w_1w_2w_1}$ and $N\bk{w_4w_5w_4}$.
	The decomposition which follows is
	\[
	\Ind_B^G\lambda_0 = \bigoplus_{X\subseteq \set{1,5}} \Ind _{P_{1,5}}^G \St_X \otimes \chi_0 .
	\]
	
	\item Choosing the $1^{st}$ and $4^{th}$ vertices in the Dynkin diagram:
	\begin{center}
		\begin{tikzpicture}
		\dynkin[open,label,parabolic=36]{A}{5}
		\dynkincloseddot{1}
		\dynkincloseddot{4}
		\end{tikzpicture}
	\end{center}
	The associated points are $\lambda_0=-\bk{w_1\cdot\omega_1+w_4\cdot\omega_4}-t\omega_3$, where $t>0$.
	The two projections in \Cref{Cor:Commuting_Projections} are the ones associated to $N\bk{w_1w_2w_1}$ and $N\bk{w_5w_4w_5}$.
	The decomposition which follows is
	\[
	\Ind_B^G\lambda_0 = \bigoplus_{X\subseteq \set{1,4}} \Ind _{P_{1,5}}^G \St_X \otimes \chi_0 .
	\]
	
	\item Choosing the $2^{nd}$ and $5^{th}$ vertices in the Dynkin diagram:
	\begin{center}
		\begin{tikzpicture}
		\dynkin[open,label,parabolic=18]{A}{5}
		\dynkincloseddot{2}
		\dynkincloseddot{5}
		\end{tikzpicture}
	\end{center}
	The associted point is $\lambda_0=-\bk{w_2\cdot\omega_2+w_5\cdot\omega_5}$, where $t>0$.
	The two projections in \Cref{Cor:Commuting_Projections} are the ones associated to $N\bk{w_2w_1w_2}$ and $N\bk{w_4w_5w_4}$.
	The decomposition which follows is
	\[
	\Ind_B^G\lambda_0 = \bigoplus_{X\subseteq \set{2,5}} \Ind _{P_{1,5}}^G \St_X \otimes \chi_0 .
	\]
\end{enumerate}
\end{Ex}

These examples shows that in order for the intertwining operators to commute, the choice of vertices $i_1$,...,$i_l$ in the diagram and the set of balls $B_{1}\bk{r}$,...,$B_{l}\bk{r}$ of radius $r$ around them should satisfy the following conditions:
\begin{enumerate}
	\item $B_j\bk{1}\setminus\set{\alpha_j}$ for any $1\leq j\leq n$.
	\item $B_j\bk{1}\cap B_k\bk{1}=\emptyset$ for all $j\neq k$.
	\item For any $j$ there exist at most one $k$ such that $B_j\bk{2}\cap B_k\bk{2}\neq \emptyset$, in which case $\coset{B_j\bk{1}\cup B_k\bk{1}}\setminus\coset{\set{\alpha_{i_j},\alpha_{i_k}}\cup\bk{B_j\bk{2}\cap B_k\bk{2}}}\neq \emptyset$.
\end{enumerate}

We now list the possible choices of vertices in the Dynkin diagrams of type $E_n$.
We also denote the different maximal choices of $S_1$ and $S_2$.
\begin{center}
	\begin{tikzpicture}
	\dynkin[open,label,parabolic=72]{E}{6}
	\dynkincloseddot{1}
	\dynkincloseddot{5}
	\end{tikzpicture}
	\begin{tikzpicture}
	\dynkin[open,label,parabolic=40]{E}{6}
	\dynkincloseddot{1}
	\dynkincloseddot{6}
	\end{tikzpicture}
	\begin{tikzpicture}
	\dynkin[open,label,parabolic=34]{E}{6}
	\dynkincloseddot{3}
	\dynkincloseddot{6}
	\end{tikzpicture}
	\\
	Type $E_6$
\end{center}

\begin{center}
	\begin{tikzpicture}
	\dynkin[open,label,parabolic=200]{E}{7}
	\dynkincloseddot{1}
	\dynkincloseddot{5}
	\end{tikzpicture}
	\begin{tikzpicture}
	\dynkin[open,label,parabolic=156]{E}{7}
	\dynkincloseddot{1}
	\dynkincloseddot{6}
	\end{tikzpicture}
	\begin{tikzpicture}
	\dynkin[open,label,parabolic=168]{E}{7}
	\dynkincloseddot{1}
	\dynkincloseddot{6}
	\end{tikzpicture}
	\begin{tikzpicture}
	\dynkin[open,label,parabolic=92]{E}{7}
	\dynkincloseddot{1}
	\dynkincloseddot{7}	
	\end{tikzpicture}
	\begin{tikzpicture}
	\dynkin[open,label,parabolic=104]{E}{7}
	\dynkincloseddot{1}
	\dynkincloseddot{7}	
	\end{tikzpicture}
	\begin{tikzpicture}
	\dynkin[open,label,parabolic=154]{E}{7}
	\dynkincloseddot{2}
	\dynkincloseddot{6}
	\end{tikzpicture}
	\begin{tikzpicture}
	\dynkin[open,label,parabolic=90]{E}{7}
	\dynkincloseddot{2}
	\dynkincloseddot{7}
	\end{tikzpicture}
	\begin{tikzpicture}
	\dynkin[open,label,parabolic=162]{E}{7}
	\dynkincloseddot{3}
	\dynkincloseddot{6}	
	\end{tikzpicture}
	\begin{tikzpicture}
	\dynkin[open,label,parabolic=150]{E}{7}
	\dynkincloseddot{3}
	\dynkincloseddot{6}	
	\end{tikzpicture}
	\begin{tikzpicture}
	\dynkin[open,label,parabolic=86]{E}{7}
	\dynkincloseddot{3}
	\dynkincloseddot{7}	
	\end{tikzpicture}
	\begin{tikzpicture}
	\dynkin[open,label,parabolic=66]{E}{7}
	\dynkincloseddot{3}
	\dynkincloseddot{7}	
	\end{tikzpicture}
	\begin{tikzpicture}
	\dynkin[open,label,parabolic=78]{E}{7}
	\dynkincloseddot{4}
	\dynkincloseddot{7}	
	\end{tikzpicture}
	\\
	Type $E_7$
\end{center}

\begin{center}
	\begin{tikzpicture}
	\dynkin[open,label,parabolic=456]{E}{8}
	\dynkincloseddot{1}
	\dynkincloseddot{5}
	\end{tikzpicture}
	\begin{tikzpicture}
	\dynkin[open,label,parabolic=412]{E}{8}
	\dynkincloseddot{1}
	\dynkincloseddot{6}
	\end{tikzpicture}
	\begin{tikzpicture}
	\dynkin[open,label,parabolic=424]{E}{8}
	\dynkincloseddot{1}
	\dynkincloseddot{6}
	\end{tikzpicture}
	\begin{tikzpicture}
	\dynkin[open,label,parabolic=316]{E}{8}
	\dynkincloseddot{1}
	\dynkincloseddot{7}	
	\end{tikzpicture}
	\begin{tikzpicture}
	\dynkin[open,label,parabolic=348]{E}{8}
	\dynkincloseddot{1}
	\dynkincloseddot{7}	
	\end{tikzpicture}
	\begin{tikzpicture}
	\dynkin[open,label,parabolic=360]{E}{8}
	\dynkincloseddot{1}
	\dynkincloseddot{7}	
	\end{tikzpicture}
	\begin{tikzpicture}
	\dynkin[open,label,parabolic=188]{E}{8}
	\dynkincloseddot{1}
	\dynkincloseddot{8}
	\end{tikzpicture}
	\begin{tikzpicture}
	\dynkin[open,label,parabolic=220]{E}{8}
	\dynkincloseddot{1}
	\dynkincloseddot{8}
	\end{tikzpicture}
	\begin{tikzpicture}
	\dynkin[open,label,parabolic=232]{E}{8}
	\dynkincloseddot{1}
	\dynkincloseddot{8}
	\end{tikzpicture}
	\begin{tikzpicture}
	\dynkin[open,label,parabolic=410]{E}{8}
	\dynkincloseddot{2}
	\dynkincloseddot{6}	
	\end{tikzpicture}
	\begin{tikzpicture}
	\dynkin[open,label,parabolic=314]{E}{8}
	\dynkincloseddot{2}
	\dynkincloseddot{7}	
	\end{tikzpicture}
	\begin{tikzpicture}
	\dynkin[open,label,parabolic=346]{E}{8}
	\dynkincloseddot{2}
	\dynkincloseddot{7}	
	\end{tikzpicture}
	\begin{tikzpicture}
	\dynkin[open,label,parabolic=186]{E}{8}
	\dynkincloseddot{2}
	\dynkincloseddot{8}	
	\end{tikzpicture}
	\begin{tikzpicture}
	\dynkin[open,label,parabolic=218]{E}{8}
	\dynkincloseddot{2}
	\dynkincloseddot{8}	
	\end{tikzpicture}
	\begin{tikzpicture}
	\dynkin[open,label,parabolic=310]{E}{8}
	\dynkincloseddot{3}
	\dynkincloseddot{7}	
	\end{tikzpicture}
	\begin{tikzpicture}
	\dynkin[open,label,parabolic=342]{E}{8}
	\dynkincloseddot{3}
	\dynkincloseddot{7}	
	\end{tikzpicture}
	\begin{tikzpicture}
	\dynkin[open,label,parabolic=354]{E}{8}
	\dynkincloseddot{3}
	\dynkincloseddot{7}	
	\end{tikzpicture}
	\begin{tikzpicture}
	\dynkin[open,label,parabolic=182]{E}{8}
	\dynkincloseddot{3}
	\dynkincloseddot{8}	
	\end{tikzpicture}
	\begin{tikzpicture}
	\dynkin[open,label,parabolic=214]{E}{8}
	\dynkincloseddot{3}
	\dynkincloseddot{8}	
	\end{tikzpicture}
	\begin{tikzpicture}
	\dynkin[open,label,parabolic=226]{E}{8}
	\dynkincloseddot{3}
	\dynkincloseddot{8}	
	\end{tikzpicture}
	\begin{tikzpicture}
	\dynkin[open,label,parabolic=302]{E}{8}
	\dynkincloseddot{4}
	\dynkincloseddot{7}	
	\end{tikzpicture}
	\begin{tikzpicture}
	\dynkin[open,label,parabolic=334]{E}{8}
	\dynkincloseddot{4}
	\dynkincloseddot{7}	
	\end{tikzpicture}
	\begin{tikzpicture}
	\dynkin[open,label,parabolic=190]{E}{8}
	\dynkincloseddot{4}
	\dynkincloseddot{8}	
	\end{tikzpicture}
	\begin{tikzpicture}
	\dynkin[open,label,parabolic=222]{E}{8}
	\dynkincloseddot{4}
	\dynkincloseddot{8}	
	\end{tikzpicture}
	\begin{tikzpicture}
	\dynkin[open,label,parabolic=158]{E}{8}
	\dynkincloseddot{5}
	\dynkincloseddot{8}	
	\end{tikzpicture}
	\\
	Type $E_8$
\end{center}

\appendix

\section{Some Facts on Root Systems and Weyl Groups}
\label{Appendix:Facts_on_roots_systems_and_Weyl_groups}
In this section we record a few simple but useful facts about the action of the Weyl group on the root system for which we weren't able to locate a convenient reference.
We retain the notations of \Cref{Section:Preliminaries}.

\begin{Lem}
\label{App_Lem_1}
Let $w\in W$ and $\alpha\in \Delta$. If $w$ and $w_\alpha$ commute, then
\begin{align*}
& w\bk{\alpha} \in \set{\alpha, -\alpha} \\
& w\bk{\check{\alpha}} \in \set{\check{\alpha}, -\check{\alpha}} .
\end{align*}
\end{Lem}

\begin{proof}
Indeed,
\begin{align*}
w_\alpha\bk{w\bk{\alpha}} &= w\bk{\alpha} - \gen{w\bk{\alpha},\check{\alpha}}\alpha \\
w\bk{w_\alpha\bk{\alpha}} &= w\bk{-\alpha} = -w\bk{\alpha} .
\end{align*}
Since $w_\alpha w = w w_\alpha$ it follows that
\[
w\bk{\alpha} = \frac{1}{2} \gen{w\bk{\alpha},\check{\alpha}} \alpha
\]
And hence, since the root system is reduced, it follows that $w\bk{\alpha}\in \set{\alpha,-\alpha}$.
Similarly $w\bk{\check{\alpha}} \in \set{\check{\alpha}, -\check{\alpha}}$.
\end{proof}

\begin{Lem}
\label{App_Lem_2}
Let $w\in W$ and $\alpha\in \Delta$.
Assume that $w$ and $w_\alpha$ commute.
Then
\[
\gen{w w_\alpha \lambda,\check{\alpha}} = \pm \gen{\lambda,\check{\alpha}} \quad\forall \lambda \in \mathfrak{a}_\R^\ast.
\]
\end{Lem}

\begin{proof}
We start by noting that $w^{-1}$ also commutes with $w_\alpha$.
Assume that $w\bk{\check{\alpha}} = \check{\alpha} = w^{-1}\bk{\check{\alpha}}$ (the case $w\bk{\check{\alpha}} = -\check{\alpha}$ follows similarly).
Hence
\begin{align*}
\gen{w w_\alpha \lambda,\check{\alpha}}
& = \gen{w_\alpha\lambda,w^{-1}\check{\alpha}} \\
& = \gen{w_\alpha\lambda, \check{\alpha}} \\
& = \gen{\lambda, w_\alpha\check{\alpha}} \\
& = - \gen{\lambda,\check{\alpha}} .
\end{align*}
\end{proof}

\begin{Lem}
	\label{App_Lem_3}
	Assume that
	\[
	\set{\lambda\in \mathfrak{a}_\R^\ast\mvert \gen{w_\alpha w w_\alpha \lambda ,\check{\alpha}} =1} =
	\set{\lambda\in \mathfrak{a}_\R^\ast\mvert \gen{\lambda,\check{\alpha}}=1} .
	\]
	Then $w$ and $w_\alpha$ commute.
\end{Lem}

\begin{proof}
	Fix $\lambda_0\in H_1=\set{\lambda\in \mathfrak{a}_\R^\ast\mvert \gen{\lambda,\check{\alpha}}=1}$ and consider the vector space $V=H_1-\lambda_0 = \alpha^\perp$.
	It follows that
	\[
	\set{\lambda\in \mathfrak{a}_\R^\ast\mvert \gen{w_\alpha w w_\alpha \lambda ,\check{\alpha}} =0} =
	\set{\lambda\in \mathfrak{a}_\R^\ast\mvert \gen{\lambda,\check{\alpha}}=0} .
	\]
	Namely, $\alpha^\perp=\bk{w_\alpha w w_\alpha \alpha}^\perp$.
	Since the root system is reduced, we conclude that $\alpha = \pm w_\alpha w w_\alpha \alpha$, or in other words $w\alpha = \pm \alpha$.
	It follows that $w\check\alpha=\pm\check\alpha$ (same sign).
	We show that $ww_\alpha=w_\alpha w$ by examining the action of both sides on $\mathfrak{a}_\R^\ast$.
	Indeed,
	\[
	\begin{split}
	ww_\alpha\lambda 
	& = w\bk{\lambda-\gen{\lambda,\check{\alpha}}\alpha} \\
	& = w\lambda \mp \gen{\lambda,\check{\alpha}}\alpha \\
	& = w\lambda \mp \gen{\lambda,\check{\alpha}}\alpha \\
	& = w\lambda - \gen{w\lambda,\check{\alpha}}\alpha = w_\alpha w\lambda .
	\end{split}
	\]
	
\end{proof}

\begin{Lem}
	\label{App_Lem_4}
	Assume that $\lambda\in \mathfrak{a}_\R^\ast$ satisfy
	$\gen{\lambda',\check{\beta}} \leq 0$ for all $\beta\in\Delta$ and $\gen{\lambda',\check{\alpha}}=-1$.
	Then 
	\begin{equation}
	\label{App_Lem_4_1}
	\gen{\lambda',\check{\beta}}\leq -\frac{1}{2}\gen{\alpha,\check{\beta}}\quad \forall \beta\in\Phi^{+}\setminus\set{\alpha} .
	\end{equation}
\end{Lem}

\begin{proof}
	For the sake of the proof, it is convenient to use the inner product on $V=\mathfrak{a}_\R^\ast$ underlying the pairing $V\times \check{V}\to\R$ given by $\gen{\cdot,\cdot}$.
	Indeed, 
	$\mathfrak{a}_\R^\ast$ is equipped with an inner product $\bk{\cdot,\cdot}$ space such that
	\[
	\gen{\gamma_1,\check{\gamma_2}} = 2\frac{\bk{\gamma_1,\gamma_2}}{\bk{\gamma_2,\gamma_2}} \quad \forall \gamma_1,\gamma_2\in\Phi.
	\]
	We note that for $\beta,\gamma\in\Delta$ it holds that
	\[
	\bk{\beta,\omega_\gamma} = \begin{cases} \frac{\bk{\beta,\beta}}{2},& \beta=\gamma \\ 0,& \beta\neq\gamma \end{cases} .
	\]
	
	The inequality $\gen{\lambda',\check{\beta}}\leq -\frac{1}{2}\gen{\alpha,\check{\beta}}$ is equivalent to
	\begin{equation}
	\label{App_Lem_4_2}
	\bk{\lambda',\beta}\leq -\frac{1}{2}\bk{\alpha,\beta} .
	\end{equation}
	\Cref{App_Lem_4} follows from:
	
	\underline{Claim:} 
	\begin{equation}
	\label{App_Lem_4_3}
	\suml_{\gamma\in\Delta\setminus\bk{S\cup\set{\alpha}}} n_\gamma\bk{\beta} \frac{\bk{\gamma,\gamma}}{2} m_\gamma\bk{\lambda'} \leq 
	\frac{n_\alpha\bk{\beta} \cdot \bk{\alpha,\alpha}-\bk{\alpha,\beta}}{2}.
	\end{equation}
	
	Indeed, \Cref{App_Lem_4_3} holds since its left-hand side is non-positive while its right-hand side is non-negative:
	\begin{itemize}
		\item By assumption, $m_\gamma\bk{\lambda'}=\bk{\lambda',\gamma}<0$ for any $\gamma\in\Delta\setminus\bk{S\cup\set{\alpha}}$ and $n_\gamma\bk{\beta}\geq 0$ for all $\gamma\in\Delta$.
		Hence, the left-hand side is non-positive.
		
		
		\item Note that
		\[
		\frac{\bk{\alpha,\beta}}{\bk{\alpha,\alpha}} = \suml_{\gamma\in\Delta} n_\gamma\bk{\beta} \frac{\bk{\alpha,\gamma}}{\bk{\alpha,\alpha}} \leq n_{\alpha}\bk{\beta} ,
		\]
		since $\bk{\alpha,\gamma}\leq 0$ for all $\gamma\in\Delta\setminus\set{\alpha}$ and $n_\gamma\bk{\beta}\geq 0$ for all $\gamma\in\Delta$.
		It follows that $n_\alpha\bk{\beta} \cdot \bk{\alpha,\alpha}-\bk{\alpha,\beta}\geq 0$.
	\end{itemize}
	
	We show that \Cref{App_Lem_4_3} is equivalent to \Cref{App_Lem_4_2}.

	Write $\displaystyle \beta = \suml_{\gamma\in\Delta} n_\gamma\bk{\beta}\gamma$ and $\displaystyle \lambda' = \suml_{\gamma\in\Delta} m_\gamma\bk{\lambda'}\omega_\gamma$.
	Also, let $S=\set{\gamma\in\Delta\mvert \bk{\lambda',\gamma}=0}$.
	Then
	\[
	\begin{split}
	\bk{\lambda',\beta} 
	& = \suml_{\gamma\in\Delta} n_\gamma\bk{\beta} \bk{\lambda',\gamma} \\
	& = \suml_{\gamma\in\Delta} n_\gamma\bk{\beta} \frac{\bk{\gamma,\gamma}}{2} m_\gamma\bk{\lambda'} \\
	& = \suml_{\gamma\in\Delta\setminus\bk{S\cup\set{\alpha}}} n_\gamma\bk{\beta} \frac{\bk{\gamma,\gamma}}{2} m_\gamma\bk{\lambda'} - n_\alpha\bk{\beta} \frac{\bk{\alpha,\alpha}}{2} \\
	\end{split}
	\]
	Plugging this into \Cref{App_Lem_4_2} yields \Cref{App_Lem_4_3}.
%
	

\end{proof}

\bibliographystyle{alpha}
\bibliography{Singularities}

\begin{thebibliography}{MW95}

\bibitem[BJ08]{MR2490651}
Dubravka Ban and Chris Jantzen.
\newblock Jacquet modules and the {L}anglands classification.
\newblock {\em Michigan Math. J.}, 56(3):637--653, 2008.

\bibitem[BW00]{MR1721403}
A.~Borel and N.~Wallach.
\newblock {\em Continuous cohomology, discrete subgroups, and representations
  of reductive groups}, volume~67 of {\em Mathematical Surveys and Monographs}.
\newblock American Mathematical Society, Providence, RI, second edition, 2000.

\bibitem[GK81]{MR620252}
S.~S. Gelbart and A.~W. Knapp.
\newblock Irreducible constituents of principal series of {${\rm SL}_{n}(k)$}.
\newblock {\em Duke Math. J.}, 48(2):313--326, 1981.

\bibitem[GK82]{MR644669}
S.~S. Gelbart and A.~W. Knapp.
\newblock {$L$}-indistinguishability and {$R$}\ groups for the special linear
  group.
\newblock {\em Adv. in Math.}, 43(2):101--121, 1982.

\bibitem[GW80]{MR597811}
Roe Goodman and Nolan~R. Wallach.
\newblock Whittaker vectors and conical vectors.
\newblock {\em J. Funct. Anal.}, 39(2):199--279, 1980.

\bibitem[HM15]{MR3267117}
Marcela Hanzer and Goran Mui\'c.
\newblock Degenerate {E}isenstein series for {$Sp(4)$}.
\newblock {\em J. Number Theory}, 146:310--342, 2015.

\bibitem[HS]{HalawiSegal}
Hezi Halawi and Avner Segal.
\newblock Structure of degenerate principal series for exceptional $p$-adic
  groups.
\newblock {\em Unpu}.

\bibitem[Kim95]{MR1351833}
Henry~H. Kim.
\newblock The residual spectrum of {${\rm Sp}_4$}.
\newblock {\em Compositio Math.}, 99(2):129--151, 1995.

\bibitem[Kim96]{MR1426903}
Henry~H. Kim.
\newblock The residual spectrum of {$G_2$}.
\newblock {\em Canad. J. Math.}, 48(6):1245--1272, 1996.

\bibitem[Kon03]{MR2050093}
Takuya Konno.
\newblock A note on the {L}anglands classification and irreducibility of
  induced representations of {$p$}-adic groups.
\newblock {\em Kyushu J. Math.}, 57(2):383--409, 2003.

\bibitem[KS71]{MR0460543}
A.~W. Knapp and E.~M. Stein.
\newblock Intertwining operators for semisimple groups.
\newblock {\em Ann. of Math. (2)}, 93:489--578, 1971.

\bibitem[KS80]{MR582703}
A.~W. Knapp and E.~M. Stein.
\newblock Intertwining operators for semisimple groups. {II}.
\newblock {\em Invent. Math.}, 60(1):9--84, 1980.

\bibitem[Lao]{LaoResidualSpectrum}
Jing~Feng Lao.
\newblock Residual spectrum of quasi-split {$Spin(8)$} defined by a cubic
  extension.
\newblock {\em Preprint}.

\bibitem[MW95]{MW_AppendixIII}
C.~M{\oe}glin and J.-L. Waldspurger.
\newblock {\em Spectral decomposition and {E}isenstein series}, volume 113 of
  {\em Cambridge Tracts in Mathematics}, chapter Appendix III, pages 298--313.
\newblock Cambridge University Press, Cambridge, 1995.
\newblock Une paraphrase de l'{\'E}criture [A paraphrase of Scripture].

\bibitem[Sega]{SegalEisen}
A.~Segal.
\newblock The degenerate {E}isenstein series attached to the {H}eisenberg
  parabolic subgroups of quasi-split forms of {$D_4$}.
\newblock {\em To appear in Tansaction of the American Mathematical Society}.

\bibitem[Segb]{SegalResiduesD4}
Avner Segal.
\newblock The degenerate residual spectrum of quasi-split forms of {$Spin_8$}
  associated to the {H}eisenberg parabolic subgroup.
\newblock {\em Preprint}.

\bibitem[Ste68]{MR0466335}
Robert Steinberg.
\newblock {\em Lectures on {C}hevalley groups}.
\newblock Yale University, New Haven, Conn., 1968.
\newblock Notes prepared by John Faulkner and Robert Wilson.

\bibitem[Win78]{MR517138}
Norman Winarsky.
\newblock Reducibility of principal series representations of {$p$}-adic
  {C}hevalley groups.
\newblock {\em Amer. J. Math.}, 100(5):941--956, 1978.

\bibitem[{\v{Z}}am97]{Zampera1997}
Sini{\v{s}}a {\v{Z}}ampera.
\newblock The residual spectrum of the group of type {$G_2$}.
\newblock {\em J. Math. Pures Appl. (9)}, 76(9):805--835, 1997.

\end{thebibliography}


\end{document}